\documentclass[11pt]{article}
\usepackage{exscale,relsize}
\usepackage{amsmath}
\usepackage{amsfonts}
\usepackage[hidelinks]{hyperref}
\usepackage{amssymb}
\usepackage{calc}
\usepackage{theorem}
\usepackage{pifont}      
\usepackage{array}
\usepackage{color}

\newcommand{\To}{\ensuremath{\rightrightarrows}}

\newcommand{\scal}[2]{\langle{{#1},{#2}}\rangle}
\newcommand{\Scal}[2]{\left\langle{{#1},{#2}}\right\rangle}

\newcommand{\HH}{\ensuremath{\mathcal H}}

\newcommand{\st}{\ensuremath{\;|\;}}

\newcommand{\RR}{\ensuremath{\mathbb R}}

\newcommand{\RX}{\ensuremath{\,\left]-\infty,+\infty\right]}}
\newcommand{\RXX}{\ensuremath{\,\left[-\infty,+\infty\right]}}

\newcommand{\NN}{\ensuremath{\mathbb N}}

\newcommand{\dom}{\ensuremath{\operatorname{dom}}}

\newcommand{\gra}{\ensuremath{\operatorname{gra}}}

\newcommand{\inte}{\ensuremath{\operatorname{int}}}

\newcommand{\reli}{\ensuremath{\operatorname{ri}}}

\newcommand{\average}{\ensuremath{\mathcal{R}_{\mu}({\bf A},{\boldsymbol \lambda})}}

\newcommand{\averageonelambda}{\ensuremath{\mathcal{R}({\bf A},{\boldsymbol \lambda})}}

\newcommand{\averageinverse}{\ensuremath{\mathcal{R}_{\mu^{-1}}({\bf A}^{-1},{\boldsymbol \lambda})}}
\newcommand{\averageoneinverse}{\ensuremath{\mathcal{R}({\bf A}^{-1},{\boldsymbol \lambda})}}

\newcommand{\ran}{\ensuremath{\operatorname{ran}}}

\newcommand{\rank}{\ensuremath{\operatorname{rank}}}

\newcommand{\Fix}{\ensuremath{\operatorname{Fix}}}

\newcommand{\Id}{\ensuremath{\operatorname{Id}}}

\newcommand{\gr}{\ensuremath{\operatorname{gra}}}

\newcommand{\minf}{\ensuremath{-\infty}}

\renewcommand{\phi}{\ensuremath{\varphi}}

\newtheorem{theorem}{Theorem}[section]

\newtheorem{fact}[theorem]{Fact}
\newtheorem{corollary}[theorem]{Corollary}
\newtheorem{proposition}[theorem]{Proposition}
\newtheorem{definition}[theorem]{Definition}

\theoremstyle{plain}{\theorembodyfont{\rmfamily}
}
\theoremstyle{plain}{\theorembodyfont{\rmfamily}
}
\theoremstyle{plain}{\theorembodyfont{\rmfamily}
}
\theoremstyle{plain}{\theorembodyfont{\rmfamily}
\newtheorem{example}[theorem]{Example}}
\theoremstyle{plain}{\theorembodyfont{\rmfamily}
\newtheorem{remark}[theorem]{Remark}}
\theoremstyle{plain}{\theorembodyfont{\rmfamily}
}

\newcommand{\pluss}{{\hskip1pt \raise1pt\vbox{\hrule width6pt \vskip1pt
\hrule width6pt}\kern-4pt{\lower1pt\hbox{\vrule height6pt \kern1pt\vrule
height6pt}}\hskip5pt}}

\begin{document}

\title{{\fontfamily{ptm}\selectfont The resolvent average of monotone operators:\\ dominant and recessive properties}}

\author{
Sedi Bartz\thanks{Mathematics, Irving K.\ Barber School,
University of British Columbia Okanagan, Kelowna, British Columbia
V1V 1V7, Canada. E-mail: \texttt{sedi.bartz@ubc.ca}.},
Heinz H.\ Bauschke\thanks{Mathematics, Irving K.\ Barber School,
University of British Columbia Okanagan, Kelowna, British Columbia
V1V 1V7, Canada. E-mail: \texttt{heinz.bauschke@ubc.ca}.},
Sarah M. Moffat\thanks{Mathematics, Irving K.\ Barber School, University
of British Columbia Okanagan, Kelowna, British Columbia V1V 1V7,
Canada. E-mail: \texttt{sarah.moffat@ubc.ca}.},
and
Xianfu Wang\thanks{Mathematics, Irving K.\ Barber School, University
of British Columbia Okanagan, Kelowna, British Columbia V1V 1V7,
Canada. E-mail: \texttt{shawn.wang@ubc.ca}.}}

\date{May 11, 2015}

\maketitle


\begin{abstract} \noindent
Within convex analysis, a rich theory  with various applications has been evolving since the proximal average of convex functions was first introduced over a decade ago. When one considers the subdifferential of the proximal average, a natural averaging operation of the subdifferentials of the averaged functions emerges. In the present paper we extend the reach of this averaging operation to the framework of monotone operator theory in Hilbert spaces, transforming it into the resolvent average. The theory of resolvent averages contains many desirable properties. In particular, we study a detailed list of properties of monotone operators and classify them as dominant or recessive with respect to the resolvent average. As a consequence, we recover a significant part of the theory of proximal averages. Furthermore, we shed new light on the proximal average and present novel results and desirable properties the proximal average possesses which have not been previously available.       
\end{abstract}




\section{Introduction}\label{intro}
The proximal average of two convex functions was first considered in \cite{BMR}. Since then, in a series of papers, the definition of the proximal average was refined and its useful properties were studied and employed in various applications revealing a rich theory with promising potential for further evolution and applications. One of the latest forms of the proximal average we refer to in the present paper is given in Definition~\ref{proximal average def} below. Some other cornerstones in the study of the proximal average include: \cite{BGLW} where many useful properties and examples where presented, \cite{BLT} where it was demonstrated that the proximal average defines a homotopy on the class of convex functions (unlike other, classical averages), and, also, a significant application \cite{BW} where the proximal average was employed in order to explicitly construct \emph{autoconjugate} representations of monotone operators, also known as \emph{self-dual} Lagrangians, the importance of which in variational analysis is demonstrated in detail in the monograph~\cite{Gho}. A recent application of the proximal average in the theory of \emph{machine learning} is~\cite{Yu}. When subdifferentiating the proximal average, we obtain an averaging operation of the subdifferentials of the underlying functions  (see equation~\eqref{resolvent average of sub} below). 

Monotone operators are fundamentally important in analysis and
optimization
\cite{AusTeb}, 
\cite{BC2011}, 
\cite{Borwein}, 
\cite{BV},
\cite{BurIus}, 
\cite{RockWets}, 
\cite{Simons2}. 
In the present paper, we analyze the resolvent average (see Definition~\ref{resolvent average def} below), which significantly extends the above averaging operation of subdifferentials to the general framework of monotone operator theory. 
(See also \cite{bmow2013}, \cite{Wang}, and \cite{Moffat} for some
earlier works on the resolvent average.)
We present powerful general properties the resolvent average
possesses and then focus on the study of more specific
inheritance properties of the resolvent average. Namely, we go through a detailed list of attractive properties of monotone operators and classify them as \emph{dominant} or \emph{recessive} with respect to the resolvent average by employing the following notions:
\begin{definition}[inheritance of properties]
Let $C$ be a set and let $I$ be an index set. Suppose that $\mathcal{AVE}:C^I\to C$. Then a property $(p)$ is said to be
\begin{enumerate}
\item \textbf{dominant} with respect to $\mathcal{AVE}$ if for each $(c_i)\in C^I$, the existence of $i_0\in I$ such that $c_{i_0}$ has property $(p)$ implies that $\mathcal{AVE}((c_i))$ has property $(p)$;
\item \textbf{recessive} with respect to $\mathcal{AVE}$ if $(p)$ is not dominant and for each $(c_i)\in C^I$, for each $i\in I$, $c_i$ having property $(p)$ implies that $\mathcal{AVE}((c_i))$ has property $(p)$.
\end{enumerate}
\end{definition}
We also provide several examples of the resolvent average (mainly in order to prove the recessive nature of several properties) of mappings which are monotone, however, which are not subdifferential operators. As a consequence, the resolvent average is now seen to be a natural and effective tool for averaging monotone operators which avoids many of the domain and range obstacles standing in front of classical averages such as the arithmetic average. The resolvent average is also seen to be an effective averaging technique when one wishes the average to posses specific properties, especially when the desired properties are dominant. When we restrict our attention to monotone linear relations, our current study extends the one in~\cite{bmw-res} where the resolvent average was considered as an average of positive semidefinite and definite matrices. When we restrict our attention to subdifferential operators, we recover a large part of the theory of the proximal average~\cite{BGLW}. Moreover, we present several novel results regarding the inheritance of desired properties of the proximal average which have not been previously available. 
In summary, \emph{the resolvent average provides a novel technique for generating new maximally monotone operators with desirable properties}.

The remaining of the paper is organized as follows:
In the remainder of Section~\ref{intro} we present the basic definitions, notations and the relations between them which we will employ throughout the paper. We also collect all preliminary facts necessary for our presentation. In Section~\ref{basic} we present basic properties of the resolvent average. In Section~\ref{dominant} we study dominant properties while Section~\ref{recessive} deals with recessive properties. Finally, in Section~\ref{neither} we consider combinations of properties, properties which are neither dominant nor recessive, and other observations and remarks.

Before we start our analysis, let us recall the following concepts and standard notation from monotone operator theory and convex analysis:
Throughout this paper, $\HH$ is a real Hilbert space with inner product $\scal{\cdot}{\cdot}$, induced norm $\|\cdot \|$, identity mapping $\Id$ and we set $q=\frac{1}{2}\|\cdot\|^2$. We denote the interior of a subset $C$ of $\HH$ by $\inte C$. Let $A:\HH \To \HH$ be a set-valued mapping. We say that $A$ is  \emph{proper} when the \emph{domain} of $A$, the set $\dom A=\{x\in\HH \st Ax\neq\varnothing\}$, is nonempty. The \emph{range} of $A$ is the set $\ran A = A(\HH)=\bigcup_{x \in \HH} Ax$, the \emph{graph} of $A$ is the set $\gra A = \{(x,u)\in \HH \times \HH \st u \in Ax\}$ and the inverse of $A$ is the mapping $A^{-1}$ satisfying $x\in A^{-1}u\Leftrightarrow u\in Ax$. $A$ is said to be \emph{monotone} if
$$
(\forall (x,u) \in \gra A)(\forall (y,v) \in \gra A)\quad \scal{x-y}{u-v} \geq 0.
$$
$A$ is said to be \emph{maximally monotone} if there exists no monotone operator $B$ such that $\gra A$ is a proper subset of $\gra B$. The \emph{resolvent} of $A$ is the mapping $J_A=(A+\Id)^{-1}$.
We say that $A$ is a linear relation if $\gra A$ is a linear subspace of $\HH \times \HH$. $A$ is said to be a maximally monotone linear relation if $A$ is both maximally monotone and a linear relation. 
The mapping $T: \HH \to \HH$ is said to be \emph{firmly nonexpansive} if
$$
(\forall x\in \HH)(\forall y\in \HH) \quad \|Tx-Ty\|^2 + \|(\Id-T)x-(\Id-T)y\|^2 \leq \|x-y\|^2.
$$
Obviously, if $T$ is firmly nonexpansive,
then it is \emph{nonexpansive}, that is, {Lipschitz continuous} with constant $1$, where a Lipschitz continuous mapping with constant $L$ is a mapping $T:\HH\to\HH$ such that
$$
(\forall x\in \HH)(\forall y\in \HH) \quad \|Tx-Ty\|\leq L\|x-y\|.
$$
The mapping $T$ is said to be a \emph{Banach contraction} if it is Lipschitz continuous with constant $L<1$. The point $x\in\HH$ is said to be a fixed point of the mapping $T:\HH\to\HH$ if $Tx=x$. The set of all fixed points of $T$ is denoted by $\Fix T$.  
The function $f:\HH\to\RX$ is said to be \emph{proper} if $\dom f=\{x\in\HH \st f(x)<\infty\}\neq\varnothing$. The \emph{Fenchel conjugate} of the function $f$ is the function $f^*$ which is defined by $f^*(u)=\sup_{x\in\HH}(\scal{u}{x}-f(x))$. The subdifferential of a proper function $f$ is the mapping $\partial f:\HH\rightrightarrows\HH$ which is defined by
$$
\partial f(x)=\big\{u\in \HH\ \big|\  f(x)+\scal{u}{y-x}\leq f(y),\ \ \forall y\in\HH\big\}.
$$ 
The \emph{indicator function} of a subset $C$ of $\HH$ is the function $\iota_C:\HH\to\RX$ which vanishes on $C$ and equals $\infty$ on $\HH\smallsetminus C$. The \emph{normal cone} operator of the set $C$ is the mapping $N_C=\partial\iota_C$. We will denote the \emph{nearest point projection} on the set $C$ by $P_C$.

We now recall the definition of the proximal average and present the definition of the resolvent average. To this end we will make use of the following additional notations: Throughout the paper we assume that $\mu\in\ ]0,\infty[$, $n\in\{1,2,\dots\}$ and $I=\{1,\ldots,n\}$. 
For every $i \in I$, let $A_i:\HH \To \HH$ be a mapping, let $f_i:\HH\to\RX$ be a function and let $\lambda_{i}>0,\ \sum_{i\in I}\lambda_i=1$ . We set: 
\begin{align*}
&{\bf A}=(A_{1},\ldots, A_{n}),\ {\bf A^{-1}}=(A^{-1}_{1},\ldots, A^{-1}_{n}),\ {\bf f}=(f_{1},\ldots, f_{n}),\ {\bf f^*}=(f^*_{1},\ldots, f^*_{n}), \\&\text{and}\ {\boldsymbol \lambda}=(\lambda_{1},\ldots, \lambda_{n}).
\end{align*}

\begin{definition}[proximal average]\label{proximal average def}
The \emph{${\boldsymbol \lambda}$-weighted proximal average} of ${\bf f}$ with parameter $\mu$ is the function $p_\mu({\bf f},{\boldsymbol \lambda}):\HH\to\RXX$ defined by
\begin{equation}\label{proximal average}
p_\mu({\bf f},{\boldsymbol \lambda})(x)=\frac{1}{\mu}\Big(-\frac{1}{2}\|x\|^2+\inf_{\sum_{i\in I}x_i=x}\sum_{i\in I}\lambda_i\big(\mu f_i(x_i/ \lambda_i)+\frac{1}{2}\|x_i/ \lambda_i\|^2\big)\Big),\ \ \ \ \ \ x\in\HH.
\end{equation}
We will simply write $p({\bf f},{\boldsymbol \lambda})$ when $\mu=1$, $p_\mu({\bf f})$ when all of the $\lambda_i$'s coincide and, finally, $p({\bf f})$ when $\mu=1$ and all of the $\lambda_i$'s coincide.   
\end{definition}

\begin{fact}\label{resolvents of subs and proximal reformulation}\emph{\cite[Proposition 4.3 and Theorem 6.7]{BGLW}}
Suppose that for each $i\in I$, $f_i:\HH\to\RX$ is proper, lower semicontinuous and convex. Then $p_\mu(\bf f,\boldsymbol{\lambda})$ is proper, lower semicontinuous, convex and for every $x\in\HH$,
\begin{equation}\label{proximal average reformulation}
p_\mu({\bf f},\boldsymbol{\lambda})(x)=\inf_{\sum_{i\in I}\lambda_j y_j=x}\sum_{i\in I}\lambda_if_i(y_i)+\frac{1}{\mu}\bigg(\Big(\sum_{i\in I}\lambda_iq(y_i)\Big)-q(x)\bigg),\ \ \ \ \ \ x\in\HH.
\end{equation}
Furthermore,
\begin{equation}\label{average of resolvents of sub}
J_{\mu\partial p_{\mu}({\bf f},{\boldsymbol \lambda})}=\sum_{i\in I}\lambda_iJ_{\mu\partial f_i}.
\end{equation}
\end{fact}
We recall that $J_{\partial f}=\text{Prox} f$ is Moreau's proximity operator (see~\cite{Mor}). Thus, we see that $\partial p_{\mu}({\bf f},{\boldsymbol \lambda})$ defines an averaging operation of the $\partial f_i$'s with the the weights $\lambda_i$ and parameter $\mu$ in the following manner:
\begin{equation}\label{resolvent average of sub}
\partial p_{\mu}({\bf f},{\boldsymbol \lambda})=\Big(\sum_{i\in I}\lambda_{i}(\partial f_{i}+\mu^{-1}\Id)^{-1}\Big)^{-1}-\mu^{-1}\Id.
\end{equation}
We now extend the reach of the averaging operation defined in \eqref{resolvent average of sub} and which is the subject matter of the present paper:
\begin{definition}[resolvent average]\label{resolvent average def}
The \emph{${\boldsymbol \lambda}$-weighted resolvent average} of ${\bf A}$ with parameter $\mu$ is defined by
\begin{equation}\label{ResAvgDef}
\average=\Big(\sum_{i\in I}\lambda_{i}(A_{i}+\mu^{-1}\Id)^{-1}\Big)^{-1}-\mu^{-1}\Id.
\end{equation}
We will simply write $\averageonelambda$ when $\mu=1$, $\mathcal{R}_\mu(\bf A)$ when all of the $\lambda_i$'s coincide and, finally, $\mathcal{R}(\bf A)$ when $\mu=1$ and all of the $\lambda_i$'s coincide.
\end{definition}
The motivation for naming~\eqref{ResAvgDef} the \emph{resolvent average} stems from the equivalence between equation \eqref{average of resolvents of sub} and equation \eqref{resolvent average of sub}, that is, from the  fact that equation \eqref{ResAvgDef} is equivalent to the equation
\begin{equation}\label{resolventidentity}
J_{\mu \average}=\sum_{i\in I}\lambda_{i}J_{\mu A_{i}}.
\end{equation}

The parameter $\mu$ has been useful in the study of the proximal average; in particular, when taking $\mu\to\infty$ or $\mu\to0^+$ one obtains classical averages of functions (see~\cite{BGLW}). We employ particular choices of the parameter $\mu$ in applications of the proximal average in the present paper as well.    

Suppose that for each $i\in I$, $f_i:\HH\to\RX$ is proper, lower semicontinuous, convex and set $A_i=\partial f_i$. By combining Definition~\ref{resolvent average def} together with equation~\eqref{resolvent average of sub} we see that
\begin{equation}\label{sub of proximal average}
\partial
p_\mu(\bold{f},\boldsymbol{\lambda})=\mathcal{R}_\mu(\boldsymbol{\partial}\bold{f},\boldsymbol{\lambda}).
\end{equation}

We end this introductory section with the following collection of
facts which we will employ in the remaining sections of the
paper. 

The next fact from~\cite{RockWets} was originally presented in the setting of finite-dimensional spaces, however, along with its proof from \cite{RockWets}, it holds in any Hilbert space.    
\begin{fact}[resolvent identity]\emph{\cite[Lemma 12.14]{RockWets}}
For any mapping $A:\HH\rightrightarrows\HH$,
\begin{equation}\label{must1}
J_A = \Id - J_{A^{-1}}.
\end{equation}
\end{fact}

\begin{fact}\emph{\cite[Proposition 6.17]{RockWets}}
\label{NCresolvent}
Let $C$ be a nonempty, closed and convex subset of $\HH$. Then
$$J_{N_C} = (\Id + N_C)^{-1} = P_C.$$
\end{fact}

\begin{fact}\emph{\cite[Proposition 4.2]{BC2011}}
\label{f:firm}
Let $T\colon \HH \to \HH$.
Then the following are equivalent:
\begin{enumerate}
\item $T$ is firmly nonexpansive.
\item $\Id-T$ is firmly nonexpansive.
\item \label{f:firm-nonexp}$2T-\Id$ is nonexpansive.
\item \label{f:firm-inequal}$(\forall x\in \HH)(\forall y\in \HH)$
$\|Tx-Ty\|^2 \leq \scal{x-y}{Tx-Ty}$.
\end{enumerate}
\end{fact}

\begin{corollary}\label{average of firmly}
Suppose that for each $i\in I$, $T_i:\HH\to\HH$ is firmly nonexpansive. Then $T=\sum_{i\in I}\lambda_i T_i$ is firmly nonexpansive. 
\end{corollary} 

\begin{proof}
Employing Fact~\ref{f:firm}, for each $i\in I$, letting $N_i=2T_i-\Id$, we see that $N_i$ is nonexpansive. Letting $N=\sum_{i\in I}\lambda_iN_i$, we see that $N$ is nonexpansive. Consequently, the mapping $T=\frac{1}{2}(N+\Id)$ is firmly nonexpansive.
\end{proof}

\begin{fact}\emph{\cite[Lemma 2.13]{BC2011}}
For each $i\in I$ let $x_i$ and $u_i$ be points in $\HH$ and
$\alpha_i\in\RR$ be such that $\sum_{i\in I}\alpha_{i}=1$. Then 
\begin{equation}\label{f:monoident}
\Scal{\sum_{i\in I}\alpha_{i}x_{i}}{\sum_{j\in I}\alpha_{j}u_{j}}+\frac{1}{2}\sum_{(i,j)\in I\times I}
\alpha_{i}\alpha_{j}\scal{x_{i}-x_{j}}{u_{i}-u_{j}}=\sum_{i\in I}\alpha_{i}\scal{x_{i}}{u_{i}}.
\end{equation}
Consequently,
\begin{equation}\label{normstrongconvex}
\bigg\|\sum_{i\in I}\alpha_{i}x_{i}\bigg\|^2=\sum_{i\in I}\alpha_{i}\|x_{i}\|^2
-\frac{1}{2}\sum_{(i,j)\in I\times I}
\alpha_{i}\alpha_{j}\|x_{i}-x_{j}\|^2.
\end{equation}
\end{fact}

\begin{corollary}\label{equality in convex combination of firmly nonexpansive}
Suppose that for each $i\in I$, $N_i:\HH\to\HH$ is nonexpansive, $T_i:\HH\to\HH$ is firmly nonexpansive and set $N=\sum_{i\in I}\lambda_i N_i$ and $T=\sum_{i\in I}\lambda_i T_i$. Let $x$ and $y$ be points in $\HH$ such that $\|Tx-Ty\|^2=\scal{x-y}{Tx-Ty}$. Then $T_i x-T_iy=Tx-Ty$ for every $i\in I$. As a consequence, the following assertions hold: 
\begin{enumerate}
\item 
If there exits $i_0\in I$ such that $T_{i_0}$ is injective, then $T$ is injective.\label{firmly injective} 
\item 
If $x$ and $y$ are points in $\HH$ such that $\|Nx-Ny\|=\|x-y\|$, then $N_i x-N_iy=Nx-Ny$ for every $i\in I$.\label{equality in convex combination of nonexpansive}
\item
\emph{\cite[Lemma 1.4]{Reich 1983}} If $\bigcap_{i\in I}\Fix N_i\neq\varnothing$, then 
$
\Fix N=\bigcap_{i\in I}\Fix N_i.
$\label{fixed points of average of firmly nonexpansive mappings}
\end{enumerate} 
\end{corollary}

\begin{proof}
By employing equality~\eqref{normstrongconvex} and then the firm nonexpansiveness of each $T_i$ we obtain
\begin{align*}
\scal{x-y}{Tx-Ty}&=\|T x-T y\|^2\\
&=\sum_{i\in I}\lambda_i\|T_ix-T_iy\|^2-\frac{1}{2}\sum_{(i,j)\in I\times I}\lambda_{i}\lambda_{j}\|(T_i x-T_i y)-(T_j x-T_j y)\|^2\\
&\leq\sum_{i\in I}\lambda_i\scal{x-y}{T_i x-T_iy}-\frac{1}{2}\sum_{(i,j)\in I\times I}\lambda_{i}\lambda_{j}\|(T_i x-T_i y)-(T_j x-T_j y)\|^2\\
&=\scal{x-y}{Tx-Ty}-\frac{1}{2}\sum_{(i,j)\in I\times I}\lambda_{i}\lambda_{j}\|(T_i x-T_i y)-(T_j x-T_j y)\|^2.
\end{align*}
Thus, we see that $\|(T_i x-T_i y)-(T_j x-T_j y)\|^2=0$ for every $i$ and $j$ in $I$. Hence, $T_i x-T_i y=Tx-Ty$ for every $i\in I$. (i): Follows immediately. (ii): For each $i\in I$ we suppose that $T_i=\frac{1}{2}(N_i+\Id)$ so that $T_i$ as well as the mapping $T=\sum_{i\in I}\lambda_i T_i=\frac{1}{2}(N+\Id)$ are firmly nonexpansive. Then the equality $\|Nx-Ny\|^2=\|x-y\|^2$ implies that $\|Tx-Ty\|^2=\scal{x-y}{Tx-Ty}$. Consequently, we see that $T_i x-T_iy=Tx-Ty$ for every $i\in I$, which, in turn, implies that $N_i x-N_iy=Nx-Ny$ for every $i\in I$. (iii): The inclusion $\Fix N\supseteq\bigcap_{i\in I}\Fix N_i$ is trivial. Now, suppose that $x\in\Fix N$ and let $y\in\bigcap_{i\in I}\Fix N_i\subseteq\Fix N$. Then $Nx-Ny=x-y$, in particular, $\|Nx-Ny\|=\|x-y\|$. Consequently, for each $i\in I$,
$
x-y=Nx-Ny=N_i x-N_i y=N_i x-y,
$ 
which implies that $x=N_i x$, as asserted by (iii). 
\end{proof}

\begin{fact}[Minty's Theorem]\emph{\cite{Minty} (see also \cite[Theorem 21.1]{BC2011})}\label{MintyThm}
Let $A: \HH \rightrightarrows \HH$ be monotone. Then
\begin{equation}\label{Minty's parametrization}
\gra A =\big\{ (J_Ax,(\Id-J_A)x)\ \big|\ x\in \ran(\Id+A)\big\}.
\end{equation}
Furthermore, $A$ is maximally monotone if and only if $\ran (\Id + A) = \HH$. 
\end{fact}

\begin{fact}
\label{f:Minty}
{\rm (\cite{EckBer} and \cite{Minty}.)}
Let $T\colon \HH \to \HH$ and let $A\colon \HH \To \HH$.
Then the following assertions hold:
\begin{enumerate}
\item
\label{f:Mintyi}
If $T$ is firmly nonexpansive,
then $B = T^{-1}-\Id$ is maximally monotone and $J_B=T$.
\item
If $A$ is maximally monotone, then $J_A$ has full domain, and is
single-valued and firmly nonexpansive, and $A=J_A^{-1}-\Id$. 
\end{enumerate}
\end{fact}

\begin{definition}[Fitzpatrick function]\emph{\cite{Fitz}}
With the mapping $A:\HH\rightrightarrows\HH$ we associate the \emph{Fitzpatrick function} $F_A:\HH\times\HH\to\RX$, defined by
\begin{equation}\label{Fitzpatrick function}
F_A(x,v)=\sup_{(z,w)\in\gra A}\big(\scal{w}{x}+\scal{v}{z}-\scal{w}{z}\big),\ \ \ \ \ \ (x,v)\in\HH\times\HH.
\end{equation}
\end{definition}

\begin{definition}[rectangular and paramonotone mappings]\label{D:para rec def}
The monotone mapping $A:\HH\rightrightarrows\HH$ is said to be \emph{rectangular}
{\rm \cite[Definition~31.5]{Simons2}} (also known as \emph{$3^*$ monotone}) if for every $x\in\dom A$ and every $v\in\ran A$ we have
\begin{equation}\label{rectangular definition}
\inf_{(z,w)\in\gr A} \scal{v-w}{x-z}>\minf,
\end{equation}
equivalently, if
\begin{equation}\label{Fitzpatrick rectangular}
\dom A\times\ran A\ \subseteq\ \dom F_A.
\end{equation}
The mapping $A:\HH\rightrightarrows\HH$ is said to be \emph{paramonotone} if whenever we have a pair of points $(x,v)$ and $(y,u)$ in $\gr A$ such that $\scal{x-y}{v-u}=0$, then $(x,u)$ and $(y,v)$ are also in $\gr A$.
\end{definition}

\begin{fact}\emph{\cite[Remark 4.11]{bbw2007}, \cite[Corollary 4.11]{bwy2012}}\label{f:pararecsame}
Let $A\in \RR^{N\times N}$ be monotone and set $A_{+} = \tfrac{1}{2}A+\tfrac{1}{2}A^\intercal$. Then the following assertions are equivalent:
\begin{enumerate}
\item $A$ is paramonotone;
\item $A$ is rectangular;
\item $\rank A=\rank A_{+}$;
\item $\ran A=\ran A_{+}$.
\end{enumerate}
\end{fact}

\begin{fact}\label{linear}\emph{\cite[I.2.3 and I.4]{Cross}}
 Let $A,B$ be linear relations. Then $A^{-1}$ and $A+B$ are linear relations.
\end{fact}

\section{Basic properties of the resolvent average}\label{basic} 
In this section we present several basic properties of the resolvent average. These will stand as the foundation of our entire discussion in the present paper and will be applied repeatedly.

\subsection{The inverse of the resolvent average}

We begin our discussion by recalling the following fact \cite[Theorem 5.1]{BGLW}: suppose that for each $i\in I$,  $f_i:\HH\to\RX$ is proper, lower semicontinuous and convex. Then 
\begin{equation}\label{proximal conjugate}
\big(p_\mu({\bf f},{\boldsymbol \lambda})\big)^*=p_{\mu^{-1}}({\bf f^*},{\boldsymbol \lambda}).
\end{equation}
Now for each $i\in I$ we set $A_i=\partial f_i$. Then $A_i^{-1}=\partial f_i^*$. Thus, recalling equation~\eqref{sub of proximal average}, we see that equation \eqref{proximal conjugate} turns into
\begin{equation}\label{proximal conjugate sub}
\big(\average\big)^{-1}=\partial \big(p_\mu({\bf f},{\boldsymbol \lambda})\big)^*=\partial p_{\mu^{-1}}({\bf f^*},{\boldsymbol \lambda})=\averageinverse,
\end{equation}
that is, we have an inversion  formula for the resolvent average in the case where we average subdifferential operators. We now aim at extending this result into our, more general, framework of the present paper. To this end, we will need the following general property of resolvents:
\begin{proposition}
Let $A:\HH\To\HH$. Then
\begin{align}
(A+\mu^{-1}\Id)^{-1}&=\big(\Id-\mu(A^{-1}+\mu\Id)^{-1}\big)\circ(\mu\Id)\label{mu resolvent identity}\\
&=\mu\big(\Id-(\mu^{-1}A^{-1}+\Id)^{-1}\big).\label{mu resolvent identity 2}
\end{align}
Consequently, if $B:\HH\To\HH$ is a mapping such that
\begin{equation}\label{mu resolvent identity character}
(A+\mu^{-1}\Id)^{-1}=\big(\Id-\mu(B+\mu\Id)^{-1}\big)\circ(\mu\Id),
\end{equation}
then $B=A^{-1}$.
\end{proposition}

\begin{proof}
For any mapping $F:\HH\To\HH$ we have $(\mu F)^{-1}=F^{-1}\circ(\mu^{-1}\Id)$. Employing this fact, the fact that $A+\mu^{-1}\Id=\mu^{-1}(\mu A+\Id)$ and the resolvent identity \eqref{must1}, we obtain the following chain of equalities:
\begin{align*}
(A+\mu^{-1}\Id)^{-1}&=J_{\mu A}\circ(\mu\Id)=(\Id-J_{(\mu A)^{-1}})\circ(\mu\Id)=(\Id-J_{A^{-1}\circ(\mu^{-1}\Id)})\circ(\mu\Id)\\
&=\Big(\Id-\big(A^{-1}\circ(\mu^{-1}\Id)+\Id\big)^{-1}\Big)\circ(\mu\Id)\\
&=\Big(\Id-\big(A^{-1}\circ(\mu^{-1}\Id)+(\mu\Id)\circ(\mu^{-1}\Id)\big)^{-1}\Big)\circ(\mu\Id)\\
&=\Big(\Id-\big((A^{-1}+\mu\Id)\circ(\mu^{-1}\Id)\big)^{-1}\Big)\circ(\mu\Id)=\big(\Id-\mu(A^{-1}+\mu\Id)^{-1}\big)\circ(\mu\Id)\\
&=\mu\Id-\mu(\mu^{-1}A^{-1}+\Id)^{-1}\circ(\mu^{-1}\Id)\circ(\mu\Id)=\mu\big(\Id-(\mu^{-1}A^{-1}+\Id)^{-1}\big).
\end{align*}
This completes the proof of \eqref{mu resolvent identity} and \eqref{mu resolvent identity 2}. Now, suppose that $B:\HH\To\HH$ is a mapping which satisfies equation \eqref{mu resolvent identity character}. Then, by employing equation \eqref{mu resolvent identity} to the right hand side of equation \eqref{mu resolvent identity character}, we arrive at  $(A+\mu^{-1}\Id)^{-1}=(B^{-1}+\mu^{-1}\Id)^{-1}$, which implies that $A+\mu^{-1}\Id=B^{-1}+\mu^{-1}\Id$. Since $\mu^{-1}\Id$ is single-valued, we conclude that $B^{-1}=A$ and complete the proof.
\end{proof}

\begin{theorem}[the inversion formula]\label{mainresult}
Suppose that for each $i\in I$, $A_{i}:\HH \To \HH$ is a set-valued mapping. Then
\begin{equation}\label{niceselfdual}
\big(\average\big)^{-1}=\averageinverse.
\end{equation}
\end{theorem}

\begin{proof}
By employing the definition (see~\eqref{ResAvgDef}) of $\average$, $\averageinverse$ in~\eqref{first application of ResDef},~\eqref{second application of ResDef} and by also employing equation~\eqref{mu resolvent identity} in~\eqref{application of mu resolvent identity} below, we obtain the following chain of equalities:

\begin{align}
\average&=\bigg(\sum_{i\in I}\lambda_i(A_i+\mu^{-1}\Id)^{-1}\bigg)^{-1}-\mu^{-1}\Id\label{first application of ResDef}\\
&=\bigg(\sum_{i\in I}\lambda_i\big(\Id-\mu(A_i^{-1}+\mu\Id)^{-1}\big)\circ(\mu\Id)\bigg)^{-1}-\mu^{-1}\Id\label{application of mu resolvent identity}\\
&=\bigg(\mu\Id-\mu\sum_{i\in I}\lambda_i(A_i^{-1}+\mu\Id)^{-1}\circ(\mu\Id)\bigg)^{-1}-\mu^{-1}\Id\nonumber\\
&=\bigg(\mu\Id-\mu\Big(\averageinverse+\mu\Id\Big)^{-1}\circ(\mu\Id)\bigg)^{-1}-\mu^{-1}\Id\label{second application of ResDef}
\end{align}
which, in turn, implies 
$$
(\average+\mu^{-1}\Id)^{-1}=\big(\Id-\mu[\averageinverse+\mu\Id]^{-1}\big)\circ(\mu\Id).
$$
Finally, letting $A=\average$ and $B=\averageinverse$, we may now apply characterization~\eqref{mu resolvent identity character} in order to obtain $A^{-1}=B$ and completes the proof.
\end{proof}

We see that, indeed, Theorem~\ref{mainresult} extends the reach of formula~\eqref{proximal conjugate sub}. Consequently, since the subdifferential of a proper, lower semicontinuous and convex function determines its antiderivative uniquely up to an additive constant, we note that, in fact, Theorem~\ref{mainresult} recovers formula~\eqref{proximal conjugate} up to an additive constant.

\subsection{Basic properties of $\average$ and common solutions to monotone inclusions}

\begin{proposition}\label{basic properties}
Suppose that for each $i\in I$, $A_i:\HH\rightrightarrows\HH$ is maximally monotone.

\begin{enumerate}
\item 
Let $x$ and $u$ be points in $\HH$. Then
\begin{equation}\label{average of shift}
\mathcal{R}_\mu((A_1-u,\dots,A_n-u),\boldsymbol\lambda)=\average-u
\end{equation}
and
\begin{equation}\label{average of shifted argument}
\mathcal{R}_\mu\Big(\big(A_1\big((\cdot)-x\big),\dots,A_n\big((\cdot)-x\big)\big),\boldsymbol\lambda\Big)=\average\big((\cdot)-x)\big).
\end{equation}
\item
Let $0<\alpha$. Then
\begin{equation}\label{scaled average}
\mathcal{R}_\mu(\alpha\bold{A},\boldsymbol\lambda)=\alpha\mathcal{R}_{\alpha\mu}(\bold{A},\boldsymbol\lambda)\ \ \ \text{in particular,}\ \ \ \ \average=\mu^{-1}\mathcal{R}(\mu\bold{A},\boldsymbol\lambda).
\end{equation}
\item
Let $A:\HH\rightrightarrows\HH$ be maximally monotone and suppose that for each $i\in I$ , $A_i=A$. Then 
\begin{equation}\label{same mapping average} 
\average=A.
\end{equation}
\end{enumerate}
\end{proposition}

\begin{proof}
 (i): For a mapping $B:\HH\rightrightarrows\HH$ we have $(B-u)^{-1}=B^{-1}((\cdot)+u)$. Thus, 
$$
J_{\mu\mathcal{R}_\mu((A_1-u,\dots,A_n-u),\boldsymbol\lambda)}=\sum_{i\in I}\lambda_iJ_{\mu(A_i-u)}=\sum_{i\in I}\lambda_i J_{\mu A_i}((\cdot)+\mu u)=J_{\mu\average}((\cdot)+\mu u).
$$
Consequently, \eqref{average of shift} follows. A similar argument also implies equation~\eqref{average of shifted argument}. (ii): Since $J_{\mu\mathcal{R}_\mu(\alpha\bold{A},\boldsymbol\lambda)}=\sum_{i\in I}\lambda_i J_{\mu\alpha A_i}=J_{\mu\alpha\mathcal{R}_{\mu\alpha}(\bold{A},\boldsymbol\lambda)} $, equation~\eqref{scaled average} follows.
(iii): Since $J_{\mu\average}=\sum_{i\in I}\lambda_i J_{\mu A_i}=J_{\mu A}$, we obtain $\average=A$.
\end{proof}

\begin{remark}
Suppose that for each $i\in I$, $f_i:\HH\to\RX$ is proper, lower semicontinuous and convex and we set $A_i=\partial f_i$. In this particular case formula~\eqref{scaled average} can be obtained by subdifferentiating the following formula \cite[Remark 4.2(iv)]{BGLW}:
\begin{equation}\label{scaled proximal average}
p_\mu({\bf f},{\boldsymbol \lambda})=\mu^{-1}p({\mu\bf
f},{\boldsymbol \lambda}). 
\end{equation}
\end{remark}

A strong motivation for studying the resolvent average stems from the fact that it also captures common solutions to monotone inclusions: 

\begin{theorem}[common solutions to monotone inclusions]
Suppose that for each $i\in I$, $A_i:\HH\rightrightarrows\HH$ is maximally monotone. Let $x$ and $u$ be points in $\HH$. If $\bigcap_{i\in I}A_i(x)\neq\varnothing$, then
\begin{equation}\label{common solution}
\average(x)=\bigcap_{i\in I}A_i(x).
\end{equation}
If $\bigcap_{i\in I}A_i^{-1}(u)\neq\varnothing$, then
\begin{equation}\label{common inverse solution}
\average^{-1}(u)=\bigcap_{i\in I}A_i^{-1}(u).
\end{equation}
\end{theorem}

\begin{proof}
First we prove that if $\bigcap_{i\in I}A_i^{-1}(0)\neq\varnothing$ then $\average^{-1}(0)=\bigcap_{i\in I}A_i^{-1}(0)$. Our assumption that $\bigcap_{i\in I}A_i^{-1}(0)\neq\varnothing$ means that $\bigcap_{i\in I}\Fix J_{A_i}\neq\varnothing$. Since for each $i\in I$, $J_{A_i}$ is nonexpansive, Corollary~\ref{equality in convex combination of firmly nonexpansive}\ref{fixed points of average of firmly nonexpansive mappings} guarantees that
$$
\Fix J_{\mu\average}=\Fix\sum_{i\in I}\lambda_i J_{\mu A_i}=\bigcap_{i\in I}\Fix J_{\mu A_i}
$$
which implies that $\big(\mu\average\big)^{-1}(0)=\bigcap_{i\in I}(\mu A_i^{-1})(0)$ and, consequently, that $\average^{-1}(0)=\bigcap_{i\in I}A_i^{-1}(0)$. Now, let $u\in\HH$. Given a mapping $A:\HH\rightrightarrows\HH$, then $A^{-1}(u)=(A-u)^{-1}(0)$.  Consequently, if $\bigcap_{i\in I}A_i^{-1}(u)\neq\varnothing$, then $\bigcap_{i\in I}(A_i-u)^{-1}(0)\neq\varnothing$. By employing equation~\eqref{average of shift} we obtain
\begin{align*}
\average^{-1}(u)&=(\average-u)^{-1}(0)=\mathcal{R}_\mu((A_1-u,\dots,A_n-u),\boldsymbol\lambda)^{-1}(0)\\
&=\bigcap_{i\in I}(A_i-u)^{-1}(0)
=\bigcap_{i\in I}A_i^{-1}(u)
\end{align*}
which completes the proof of equation~\eqref{common inverse solution}. Let $x\in\HH$. If $\bigcap_{i\in I}A_i(x)\neq\varnothing$, then by employing Theorem~\ref{mainresult} we now obtain
$$
\average(x)=\big(\average^{-1}\big)^{-1}(x)=\big(\averageinverse\big)^{-1}(x)=\bigcap_{i\in I}(A_i^{-1})^{-1}(x)=\bigcap_{i\in I}A_i(x)
$$
which completes the proof of equation~\eqref{common solution}.
\end{proof}

\begin{example} 
 (\textbf{convex feasibility problem}) Suppose that for each $i\in I$, $C_i$ is a nonempty, closed and convex subset of $\HH$  and set $A_i=N_{C_i}$. If $\bigcap_{i\in I}C_i\neq\varnothing$, then
$$
\average^{-1}(0)=\bigcap_{i\in I}C_i.
$$
\end{example}

\subsection{Monotonicity, domain, range and the graph of the resolvent average}

We continue our presentation of general properties of the resolvent average by focusing our attention on monotone operators.

\begin{theorem}\label{neededlater}
Suppose that for each $i\in I$, $A_i:\HH\rightrightarrows\HH$ is a set-valued mapping. Then $\average$ is maximally monotone if and only if for each $i\in I$, $A_i$ is maximally monotone. In this case
\begin{equation}\label{resolvents graph}
\gr \mu \average \subseteq \sum_{i\in I}\lambda_{i}\gr \mu A_{i},
\end{equation}
and, consequently,
\begin{equation}\label{domain and range containment}
\ran \average\subseteq\sum_{i\in I}\lambda_{i}\ran A_{i}\ \ \ \ \text{and}\ \ \ \ \dom\average\subseteq\sum_{i\in I}\lambda_{i}\dom A_{i}.
\end{equation}
\end{theorem}

\begin{proof}
By employing equation~\eqref{resolventidentity} we see that 
\begin{equation}\label{montreal}
\dom J_{\mu\average}= \bigcap_{i\in I}\dom J_{\mu A_{i}}.
\end{equation}
As a consequence, we see that $\ran(\Id+\mu\average)=\HH$ if and only if for each $i\in I$, $\ran(\Id+\mu A_i)=\HH$. Recalling Minty's Theorem (Fact~\ref{MintyThm}), we see that $\mu\average$ is maximally monotone if and only if for each $i\in I$, $\mu A_i$ is maximally monotone, that is, $\average$ is maximally monotone if and only if for each $i\in I$, $A_i$ is maximally monotone.
Finally, applying Minty's parametrization \eqref{Minty's parametrization} we obtain
\begin{align*}
\gra \mu\average &=\Big\{(J_{\mu \average}(x),\  x-J_{\mu \average}x)\ \Big|\ x\in\dom J_{\mu\average}\Big\}\\
&= \bigg\{\Big(\sum_{i\in I}\lambda_i J_{\mu A_{i}}x,\ \sum_{i\in I}\lambda_i(\Id -J_{\mu A_{i}})x\Big)\ \bigg|\ x \in \bigcap_{i\in I}\dom J_{\mu A_{i}}\bigg\}\\
&=\bigg\{\sum_{i\in I}\lambda_i\big( J_{\mu A_{i}}x,\ (\Id -J_{\mu A_{i}})x\big)\ \bigg|\ x \in \bigcap_{i\in I}\dom J_{\mu A_{i}}\bigg\}\ \subseteq\ \sum_{i\in I} \lambda_i \gra \mu A_i.
\end{align*}
which implies inclusion~\eqref{resolvents graph}.
\end{proof}

As an example, we now discuss the case where we average a monotone mapping with its inverse. It was observed in \cite[Example 5.3]{BGLW} that when $f:\HH\to\RX$ is proper, lower semicontinuous and convex, then 
\begin{equation}\label{proximal average of function with its conjugate}
p(f,f^*)=q. 
\end{equation}
Employing equation~\eqref{sub of proximal average} we see that 
$$
\mathcal{R}(\partial f,\partial f^*)=\partial p(f,f^*)=\partial q=\Id.
$$
We now extend the reach of this fact in order for it to hold in the framework of the resolvent average of monotone operators. To this end we first note the following fact:

\begin{proposition}\label{self dual is I}
Let $R:\HH\rightrightarrows\HH$ be a set-valued and monotone mapping such that $R=R^{-1}$. Then $
\gr R\subseteq \gr\Id$. Consequently, if $R$ is maximally monotone, then $R=\Id$.
\end{proposition}
\begin{proof}
Suppose that $(x,u)\in\gr R$, then we also have $(u,x)\in\gr R$. Because of the monotonicity of $R$ we now have $0\leq\scal{x-u}{u-x}=-\|x-u\|^2\leq 0$, that is, $x=u$. 
\end{proof}

\begin{corollary}\label{resolvent of mapping with inverse}
The mapping $A:\HH\rightrightarrows \HH$ is maximally monotone if and only if $\mathcal{R}(A,A^{-1})=\Id$.
\end{corollary}

\begin{proof}
If $\mathcal{R}(A,A^{-1})=\Id$, then since $\Id$ is maximally monotone, it is clear from Theorem~\ref{neededlater} that $A$ is maximally monotone. Conversely, suppose that the mapping $A:\HH\rightrightarrows\HH$ is maximally monotone. Then Theorem~\ref{mainresult} guarantees that $\big(\mathcal{R}(A,A^{-1})\big)^{-1}=\mathcal{R}(A^{-1},A)=\mathcal{R}(A,A^{-1})$ while Theorem~\ref{neededlater} guarantees that $\mathcal{R}(A,A^{-1})$ is maximally monotone. Consequently, Proposition~\ref{self dual is I} implies that $\mathcal{R}(A,A^{-1})=\Id$.
\end{proof}

Now we address domain and range properties of the resolvent average. A natural question is whether more precise Formulae than Formulae~\eqref{domain and range containment} can be attained. A precise formula for the domain of the proximal average of functions is \cite[Theorem 4.6]{BGLW}. It asserts that given proper, convex and lower semicontinuous functions $f_i:\HH\to\RX,\ i\in I$, then
$$
\dom p_\mu({\bf f},{\boldsymbol \lambda})=\sum_{i\in I}\lambda_i\dom f_i.
$$
As we shall see shortly, such precision does not hold in general for the resolvent average. However, we now aim at obtaining \emph{nearly} precise formulae for the domain and range of the resolvent average. To this end, we first recall that the \emph{relative interior} of a subset $C$ of $\HH$, which is denoted by $\reli C$, is the subset of $\HH$ which is obtained by taking the interior of $C$ when considered a subset of its closed affine hull. We say that the two subsets $C$ and $D$ of $\HH$ are \emph{nearly equal} if $\overline{C}=\overline{D}$ and $\reli C=\reli D$. In this case we write $C\simeq D$. In order to prove our near precise domain and range formulae we will employ the following fact which is an extension (to the case where we sum arbitrary finitely many rectangular mappings) of the classical and far reaching result of Brezis and Haraux \cite{BH} regarding the range of the sum of two rectangular mappings:
\begin{fact}[Pennanen]{\rm \cite[Corollary~6]{Pen}}\label{Pennanen}
Suppose that for each $i\in I$, $A_i:\HH\rightrightarrows\HH$ is monotone and rectangular. If $A=\sum_{i\in I}A_i$ is maximally monotone, then
\begin{equation}\label{BH near equality}
\ran\sum_{i\in I}A_i\simeq\sum_{i\in I}\ran{A_i}.
\end{equation}
\end{fact}

\begin{theorem}[domain and range of $\mathcal{R}_\mu$]\label{domain and range of the resolvent average}
Suppose that for each $i\in I$,  $A_i:\HH\rightrightarrows\HH$ is maximally monotone. Then 
\begin{equation}\label{finite dimension domain and range of the resolvent average}
\ran\average\simeq\sum_{i\in I}\lambda_i\ran A_i,\ \ \ \ \ \ \ \dom\average\simeq\sum_{i\in I}\lambda_i\dom A_i.
\end{equation}
\end{theorem}

\begin{proof}
We first recall that a firmly nonexpansive mapping $T:\HH\to\HH$ is maximally monotone (see \cite[Example~20.27]{BC2011}) and rectangular (see \cite[Example~24.16]{BC2011}). Thus, we see that all of the mappings $J_{\mu A_i},\ i\in I$ as well as the firmly nonexpansive mapping $J_{\mu\average}=\sum_{i\in I}\lambda_iJ_{\mu A_i}$ are maximally monotone and rectangular. As a consequence, we may now apply Fact~\ref{Pennanen} in order to obtain
$$
\ran J_{\mu\average}\simeq\sum_{i\in I}\ran\lambda_i J_{\mu A_i}=\sum_{i\in I}\lambda_i\ran J_{\mu A_i}.
$$
Since, given a maximally monotone mapping $A:\HH\rightrightarrows\HH$, we have $\dom A=\dom\mu A=\ran J_{\mu A}$ (see Minty's parametrization~\eqref{Minty's parametrization}), we arrive at the domain formula in~\eqref{finite dimension domain and range of the resolvent average}. We now combine the domain near equality with Theorem~\ref{mainresult} in order to obtain 
$$
\ran\average=\dom(\average)^{-1}=\dom\averageinverse\simeq\sum_{i\in I}\lambda_i\dom A_i^{-1}=\sum_{i\in I}\lambda_i\ran A_i.
$$
\end{proof}

We see that by employing the resolvent average we avoid constraint qualifications we had while employing the arithmetic average, one of the most obvious of which is we can now average mappings the domains of which do not intersect. 

At this point we demonstrate why even in the finite-dimensional case the near equality in~\eqref{finite dimension domain and range of the resolvent average} cannot be replaced by an equality. To this end we consider the following example.

\begin{example} \cite[Example~5.4]{BBBRW}  discussed the function $f:\RR^2\to\RX$, defined by
$$
f(x,y)=
\begin{cases}
-\sqrt{xy}, & x\geq 0,y\geq 0; \\ 
+\infty, & \text{otherwise.}
\end{cases}
$$
In \cite{BBBRW} the function $f$ was presented as an example of a proper, lower semicontinuous and sublinear function which is not subdifferentiable at certain points of the boundary of its domain, where the boundary is $(\RR_+\times\{0\})\cup(\{0\}\times\RR_+)$. In fact, it was observed in \cite{BBBRW} that the only point on the boundary of $\dom\partial f$ which belongs to $\dom\partial f$ is the origin. We now consider the function $g:\RR^2\to\RX$ defined by
 $$
g(x,y)=\max\{f(1-x,y),f(1+x,y)\}=
\begin{cases}
-\sqrt{(1-|x|)y}, & -1\leq x\leq 1 ,0\leq y; \\ 
+\infty, & \text{otherwise.}
\end{cases}
$$
We set $D=\inte\dom g=\{(x,y)|\ -1<x<1,\ 0<y\}$. Then it follows that $g$ is lower semicontinuous, convex and $\dom\partial g=D\cup\{(-1,0),(1,0)\}$. Now, we set $n=2$, $A_1=A_2=\partial g,\ \mu=1,\ 0<\lambda<1,\ \lambda_1=\lambda$ and $\lambda_2=1-\lambda$, then $\averageonelambda=\partial g$ (see formula~\eqref{same mapping average}) and hence
\begin{align*}
\dom\averageonelambda&=\dom\partial g=D\cup\{(-1,0),(1,0)\}\\
&\neq D\cup\{(-1,0),(1-2\lambda,0),(2\lambda-1,0),(1,0)\}=\lambda_1\dom A_1+\lambda_2\dom A_2.
\end{align*} 
Letting $A_1=A_2=\partial g^*=(\partial g)^{-1}$ yields the same inequality with ranges instead of domains. Finally, we note that equality fails already in~\eqref{BH near equality}, that is, since $\partial g^*$ is the subdifferential of a proper, lower semicontinuous and convex function, it is rectangular (see \cite[Example 24.9]{BC2011}), maximally monotone and we have 
$$
\ran(\partial g^*+\partial g^*)=2D\cup\{(-2,0),(2,0)\}\neq2D\cup\{(-2,0),(0,0),(2,0)\}=\ran\partial g^*+\ran\partial g^*.
$$ 
(For another example of this type, see~\cite[Example 3.14]{BM}).
\end{example}

\subsection{The Fitzpatrick function of the resolvent average}

We relate the Fitzpatrick function of the resolvent average with the Fitzpatrick functions of the averaged mappings in the following result:

\begin{theorem}[Fitzpatrick function of $\mathcal{R}_\mu$]
Suppose that for each $i\in I$,  $A_i:\HH\rightrightarrows\HH$ is maximally monotone. Then 
\begin{equation}\label{resolvent average Fitzpatrick function inequality}
F_{\mu\average}\leq\sum_{i\in I}\lambda_i F_{\mu A_i}\ \ \ \text{in particular},\ \ \ F_{\averageonelambda}\leq\sum_{i\in I}\lambda_i F_{ A_i}
\end{equation}
and
\begin{equation}\label{resolvent average Fitzpatrick function domain}
\sum_{i\in I}\lambda_i\dom F_{\mu A_i}\subseteq\dom F_{\mu\average}\ \ \ \text{in particular},\ \ \ \sum_{i\in I}\lambda_i\dom F_{ A_i}\subseteq\dom F_{\averageonelambda}.
\end{equation}
\end{theorem}

\begin{proof}
For each $\in I$, let $x_i,u_i$ and $z$ be points in $\HH$ and let $T_i=J_{\mu A_i}$. We set $(x,u)=\sum_{i\in I}\lambda_i(x_i,u_i)$ and $T=\sum_{i\in I}\lambda_i T_i=J_{\mu \average}$. Then we employ equation~\eqref{f:monoident} in order to obtain
\begin{align}
&\scal{z-Tz}{x}+\scal{u}{Tz}-\scal{z-Tz}{Tz}
=\sum_{i\in I}\lambda_i\big(\scal{z-T_i z}{x_i}+\scal{u_i}{T_iz}-\scal{z-T_iz}{T_iz}\big)\nonumber\\
&-\sum_{(i,j)\in I\times I}\frac{\lambda_i\lambda_j}{2}\big(\scal{T_jz-T_iz}{x_i-x_j}+\scal{u_i-u_j}{T_iz-T_jz}-\scal{T_jz-T_iz}{T_iz-T_jz}\big)\nonumber\\
=&\sum_{i\in I}\lambda_i\big(\scal{z-T_i z}{x_i}+\scal{u_i}{T_iz}-\scal{z-T_iz}{T_iz}\big)\nonumber\\
&-\sum_{(i,j)\in I\times I}\frac{\lambda_i\lambda_j}{2}\bigg(-\frac{1}{4}\|(u_i-u_j)-(x_i-x_j)\|^2+\Big\|\frac{(u_i-u_j)-(x_i-x_j)}{2}+(T_iz-T_jz)\Big\|^2\bigg)\nonumber\\
\leq&\sum_{i\in I}\lambda_i\big(\scal{z-T_i z}{x_i}+\scal{u_i}{T_iz}-\scal{z-T_iz}{T_iz}\big)+\sum_{(i,j)\in I\times I}\frac{\lambda_i\lambda_j}{8}\|(u_i-u_j)-(x_i-x_j)\|^2.\label{Fitzpatrick inequality}
\end{align}
Given a maximally monotone mapping $A:\HH\rightrightarrows\HH$, combining the definition \eqref{Fitzpatrick function} of the Fitzpatrick function $F_A$ with Minty's parametrization~\eqref{Minty's parametrization} implies that 
$$
F_A (x,u)=\sup_{z\in\HH}\big(\scal{z-J_A z}{x}+\scal{u}{J_A z}-\scal{z-J_A z}{J_A z}\big).
$$ 
Thus, by employing inequality~\eqref{Fitzpatrick inequality} we obtain
\begin{equation}\label{Fitzpatrick function inequality}
F_{\mu\average}(x,u)\leq\sum_{i\in I}\lambda_i F_{\mu A_i}(x_i,u_i)+\sum_{(i,j)\in I\times I}\frac{\lambda_i\lambda_j}{8}\|(u_i-u_j)-(x_i-x_j)\|^2.
\end{equation}
For each $i\in I$, letting $(x_i,u_i)=(x,u)$ in inequality~\eqref{Fitzpatrick function inequality} we arrive at inequality~\eqref{resolvent average Fitzpatrick function inequality}. For each $i\in I$, letting $(x_i,u_i)\in\dom F_{\mu A_i}$ in inequality~\eqref{Fitzpatrick function inequality} we see that $F_{\mu\average}(x,u)<\infty$, that is, we obtain inclusions~\eqref{resolvent average Fitzpatrick function domain} and complete the proof. 
\end{proof}

\section{Dominant properties of the resolvent average}\label{dominant}

\subsection{Domain and range properties}
The following domain and range properties of the resolvent average are immediate consequences of Theorem~\ref{domain and range of the resolvent average}:  

\begin{theorem}\label{t:FDdom}
\emph{\textbf{(nonempty\ interior\ of \ the\ domain, \ fullness\ of\ the\ domain\ and\ surjectivity\ are\ dominant)}}
Suppose that for each $i\in I$, $A_i:\HH\rightrightarrows\HH$ is maximally monotone.
\begin{enumerate}
\item\label{t:NEdom} If there exists $i_0\in I$ such that $\inte\dom A_{i_0}\neq\varnothing$, then $\inte\dom\average\neq\varnothing$.
\item\label{t:Fdom} If there exists $i_0\in I$ such that $\dom A_{i_0}=\HH$, then $\dom\average=\HH$.
\item\label{t:surjdom} If there exists $i_0\in I$ such that $A_{i_0}$ is surjective, then $\average$ is surjective. 
\end{enumerate}
\end{theorem}

\begin{corollary}\label{t:FDdom functions}
\emph{\textbf{(nonempty\ interior\ and\ fullness\ of\ domain\ are\ dominant w.r.t. $p_\mu$)}}
Suppose that for each $i\in I$, $f_i:\HH\to\RX$ is proper, lower semicontinuous and convex. If there exists $i_0\in I$ such that $\inte\dom f_{i_0}\neq\varnothing$, then $\inte\dom p_\mu({\bf f},{\boldsymbol \lambda})\neq\varnothing$. If there exists $i_0\in I$ such that $\dom f_{i_0}=\HH$, then $\dom p_\mu({\bf f},{\boldsymbol \lambda})=\HH$.
\end{corollary}

\begin{proof}
For each $i\in I$, we set $A_i=\partial f_i$. Suppose that for some $i_0\in I$, $\inte\dom f_{i_0}\neq\varnothing$. Since $\inte\dom f_{i_0}\subseteq\dom\partial f_{i_0}$, we see that $\inte\dom A_{i_0}\neq\varnothing$ and it now follows from Theorem~\ref{t:FDdom}\ref{t:NEdom} that $\inte\dom\average\neq\varnothing$. Now equation~\eqref{sub of proximal average} implies that $\varnothing\neq\inte\dom\average=\inte\dom\partial p_\mu({\bf f},{\boldsymbol \lambda})\subseteq\inte\dom p_\mu({\bf f},{\boldsymbol \lambda})$. A similar argument implies that if $\dom f_{i_0}=\HH$, then $\dom p_\mu({\bf f},{\boldsymbol \lambda})=\HH$.  
\end{proof}

We recall that the function $f:\HH\to\RX$ is said to be \emph{coercive} if $\lim_{\|x\|\to\infty}f(x)=\infty$.  The function $f$ is said to be \emph{supercoercive} if $f/\|\cdot\|$ is coercive. As a consequence of Theorem~\ref{t:FDdom}\ref{t:surjdom} in finite-dimensional spaces we obtain the following result:

\begin{corollary}
\emph{\textbf{(supercoercivity is dominant w.r.t.$\ p_\mu$ in $\RR^n$)}}
{\rm \cite[Lemma~3.1(iii)]{GHW}}
Suppose that $\HH$ is finite-dimensional and that for each $i\in I$, $f_i:\HH\to\RX$ is proper, lower semicontinuous and convex. If there exists $i_0\in I$ such that $f_{i_0}$ is supercoercive, then $p_\mu({\bf f},{\boldsymbol \lambda})$ is supercoercive.
\end{corollary}

\begin{proof}
For each $i\in I$, we set $A_i=\partial f_i$. We now recall that in finite-dimensional spaces, a proper, lower semicontinuous and convex function $f$ is supercoercive if and only if $\dom f^*=\HH$, which is equivalent to $\HH=\dom\partial f^*=\ran\partial f$ (combine \cite[Corollary 13.3.1]{Rock} with \cite[Corollary 14.2.2]{Rock}, or, alternatively, see  \cite[Proposition 2.16]{BB1997}). Thus, since $\ran A_{i_0}=\HH$, then Theorem~\ref{t:FDdom}\ref{t:surjdom} together with equation
\eqref{sub of proximal average} guarantee that $\ran\partial p_\mu({\bf f},{\boldsymbol \lambda})=\ran\average=\HH$. Consequently, we see that $p_\mu({\bf f},{\boldsymbol \lambda})$ is supercoercive. 
\end{proof}

\subsection{Single-valuedness}
We shall say that the mapping $A:\HH\rightrightarrows\HH$ is at most single-valued if for every $x\in\HH$, $Ax$ is either empty or a singleton.

\begin{theorem}[single-valudeness is dominant]\label{single valued}
Suppose that for each $i\in I$, $A_i:\HH\rightrightarrows\HH$ is maximally monotone. If there exists $i_0\in I$ such that $A_{i_0}$ is at most single-valued, then $\average$ is at most single-valued.
\end{theorem}

\begin{proof}
We recall that a maximally monotone mapping is at most single-valued if and only if its resolvent is injective (see \cite[Theorem 2.1(iv)]{bmw12}). Thus, if $A_{i_0}$ is at most single-valued, then $\mu A_{i_0}$ is at most single-valued and $J_{\mu A_{i_0}}$ is injective. We now apply Corollary~\ref{equality in convex combination of firmly nonexpansive}\ref{firmly injective} in order to conclude that $ J_{\mu\average}=\sum_{i\in I}\lambda_{i}J_{\mu A_{i}}$ is injective and, consequently, that $\mu\average$ is at most single-valued and so is $\average$.
\end{proof}

Recall that the proper, lower semicontinuous and convex function $f:\HH\to\RX$ is said to be \emph{essentially smooth} if the interior of its domain is nonempty and if $\partial f$ is at most single-valued. This definition of essential smoothness coincides with the classical notions of the same name in finite-dimensional spaces (see the paragraph below preceding Corollary~\ref{ess conv}). In terms of essential smoothness of the proximal average, we recover the following result:  

\begin{corollary}[essential smoothness is dominant w.r.t. $p_\mu$]\emph{\cite[Corollary 7.7]{BGLW}}\label{ess smooth}
Suppose that for each $i\in I$, $f_i:\HH\to\RX$ is proper, lower semicontinuous and convex. If there exists $i_0\in I$ such that $f_{i_0}$ is essentially smooth, then $p_\mu({\bf f},{\boldsymbol \lambda})$ is essentially smooth.
\end{corollary}
\begin{proof}
For each $i\in I$, we set $A_i=\partial f_i$. Suppose that for some $i_0\in I$, $f_{i_0}$ is essentially smooth, then $A_{i_0}=\partial f_{i_0}$ is at most single-valued and $\inte\dom A_{i_0}\neq\varnothing$. It follows follows from Corollary~\ref{t:FDdom functions} that $\inte\dom p_\mu({\bf f},{\boldsymbol \lambda})\neq\varnothing$. Furthermore, Theorem~\ref{single valued} guarantees that $\partial p_\mu({\bf f},{\boldsymbol \lambda})=\average$ is at most single-valued. Consequently, $p_\mu({\bf f},{\boldsymbol \lambda})$ is essentially smooth.
\end{proof}

\subsection{Strict monotonicity}
Recall that $A:\HH\rightrightarrows\HH$ is said to be \emph{strictly monotone} if whenever $u\in Ax$ and $v\in Ay$ are such that $x\neq y$, then $0<\scal{u-v}{x-y}$. The resolvent perspective of this issue is addressed in the following proposition.

\begin{proposition}\label{averaged strictly firmly} 
Suppose that for each $i\in I$, $T_i:\HH\to\HH$ is firmly nonexpansive and set $T=\sum_{i\in I}\lambda_i T_i$. If there exits $i_0\in I$ such that 
\begin{equation}\label{strictly firmly definition}
T_{i_0}x\neq T_{i_0}y\ \ \ \ \ \Rightarrow\ \ \ \ \ \|T_{i_0}x-T_{i_0}y\|^2<\scal{x-y}{T_{i_0}x-T_{i_0}y},
\end{equation}
then $T$ has property~\eqref{strictly firmly definition} as well. 
\end{proposition}

\begin{proof}
Suppose that $x$ and $y$ are points in $\HH$ such that $\|Tx-Ty\|^2 = \scal{x-y}{Tx-Ty}$. Then Corollary~\ref{equality in convex combination of firmly nonexpansive} implies that $Tx-Ty=T_{i_0} x-T_{i_0} y$. In particular, we see that
$$
\|T_{i_0}x-T_{i_0}y\|^2=\|Tx-Ty\|^2 = \scal{x-y}{Tx-Ty}=\scal{x-y}{T_{i_0}x-T_{i_0}y}.
$$
Consequently, since $T_{i_0}$ has property~\eqref{strictly firmly definition}, we see that $Tx-Ty=T_{i_0} x-T_{i_0} y=0$, as claimed. \end{proof}

\begin{theorem}\label{t:strmonodom}
\textbf{\emph{(strict monotonicity is dominant)}}
Suppose that for each $i\in I$, $A_i:\HH\rightrightarrows\HH$ is maximally monotone. If there exists $i_0\in I$ such that $A_{i_0}$ is strictly monotone, then $\average$ is strictly monotone.
\end{theorem}

\begin{proof}
We recall that the maximally monotone mapping $A:\HH\rightrightarrows\HH$ is strictly monotone if and only if $J_A$ has property~\eqref{strictly firmly definition} (see \cite[Theorem 2.1(vi)]{bmw12}). Thus, since $J_{\mu A_{i_0}}$ has property~\eqref{strictly firmly definition}, then Proposition~\ref{averaged strictly firmly} guarantees that $J_{\mu\average}=\sum_{i\in I}\lambda_{i}J_{\mu A_{i}}$ has property~\eqref{strictly firmly definition}, which, in turn, implies that $\mu\average$ is strictly monotone and, therefore, so is $\average$.
\end{proof}

Recall that the proper, lower semicontinuous and convex function $f:\HH\to\RX$ is said to be \emph{essentially strictly convex} if $f^*$ is essentially smooth; $f$ is said to be \emph{Legendre} if $f$ is both, essentially smooth and essentially strictly convex. These definitions of essential strict convexity and Legendreness coincide with the classical notions of the same names in finite-dimensional spaces (see the next paragraph). Now suppose that for each $i\in I$, $f_i:\HH\to\RX$ is proper, lower semicontinuous and convex. We suppose further that there exists $i_0\in I$ such that $f_{i_0}$ is essentially strictly convex. Then $f^*_{i_0}$ is essentially smooth. Consequently, Corollary~\ref{ess smooth} guarantees that $p_{\mu^{-1}}({\bf f^*},{\boldsymbol \lambda})$ is essentially smooth. Thus, according to formula~\eqref{proximal conjugate}, the function $p_\mu({\bf f},{\boldsymbol \lambda})=(p_\mu({\bf f},{\boldsymbol \lambda}))^{**}=(p_{\mu^{-1}}({\bf f^*},{\boldsymbol \lambda}))^*$ is essentially convex. Consequently, we see that essential strict convexity is dominant w.r.t. the proximal average. This line of proof of this fact was carried out in \cite{BGLW}. Since in the present paper formula~\eqref{proximal conjugate}, up to an additive constant, was recovered by Theorem~\ref{mainresult}, and since essential strict convexity is not affected by the addition of a constant to the function, we see that our discussion here does, indeed, recover the dominance of essential strict convexity w.r.t. the proximal average.

Classically, when $\HH$ is finite-dimensional, a different path that leads to the same conclusion is now available. Indeed, in this case, recall that the proper, lower semicontinuous and convex function $f:\HH\to\RX$ is said to be essentially smooth if the interior of its domain is nonempty, if $f$ is G\^{a}teaux differentiable there and the norm of its G\^{a}teaux gradient blows up as we tend from the interior to a point on the boundary of its domain. $f$ is said to be essentially strictly convex if $f$ is strictly convex on every convex subset of $\dom\partial f$, which is equivalent to $\partial f$ being strictly monotone (see \cite[Theorem 12.17]{RockWets}). Thus, given that $f_{i_0}$ is essentially strictly convex, then $\partial f_{i_0}$ is strictly monotone. Setting $A_i=\partial f_i$ for every $i\in I$, Theorem~\ref{t:strmonodom} guarantees that $\partial p_\mu({\bf f},{\boldsymbol \lambda})=\average$ is strictly monotone and, consequently, that $p_\mu({\bf f},{\boldsymbol \lambda})$ is essentially strictly convex, as asserted. Summing up both of these discussions, we have recovered the following result:

\begin{corollary}\label{ess conv}\textbf{\emph{(essential strict convexity is dominant w.r.t. $p_\mu$)}} \emph{\cite[Corollary 7.8]{BGLW}}
Suppose that for each $i\in I$, $f_i:\HH\to\RX$ is proper, lower semicontinuous and convex. If there exists $i_0\in I$ such that $f_{i_0}$ is essentially strictly convex, then $p_\mu({\bf f},{\boldsymbol \lambda})$ is essentially strictly convex.
\end{corollary}

Combining Corollary~\ref{ess smooth} together with Corollary~\ref{ess conv}, we recover the following result:

\begin{corollary}[Legendreness is dominant w.r.t. $p_\mu$]\emph{\cite[Corollary 7.9]{BGLW}}\label{Leg}
Suppose that for each $i\in I$, $f_i:\HH\to\RX$ is proper, lower semicontinuous and convex. If there exists $i_1\in I$ such that $f_{i_1}$ is essentially smooth and there exists $i_2\in I$ such that $f_{i_2}$ is essentially strictly convex, then $p_\mu({\bf f},{\boldsymbol \lambda})$ is Legendre. In particular, if there exists $i_0\in I$ such that $f_{i_0}$ is Legendre, then $p_\mu({\bf f},{\boldsymbol \lambda})$ is Legendre.
\end{corollary}

\subsection{Uniform monotonicity}

We say that the mapping $A:\HH\rightrightarrows\HH$ is monotone with modulus $\phi:[0,\infty[\to[0,\infty]$ if for every two points $(x,u)$ and $(y,v)$ in $\gr A$,
$$
\phi\big(\|x-y\|\big)\leq\scal{u-v}{x-y}.
$$  
Clearly, if $\gra A\neq\varnothing$, then $\phi(0)=0$. We recall that the mapping $A:\HH\rightrightarrows\HH$ is said to be \emph{uniformly monotone} with modulus $\phi$, $\phi$-uniformly monotone for short, if $A$ is monotone with modulus $\phi$ and $\phi(t)=0 \Leftrightarrow t=0$. We say that the mapping $T:\HH\to\HH$ is firmly nonexpansive with modulus $\phi:[0,\infty[\to[0,\infty]$ if for every pair of points $x$ and $y$ in $\HH$,
$$
\|Tx-Ty\|^2+\phi\big(\|Tx-Ty\|\big)\leq\scal{Tx-Ty}{x-y}.
$$
We also recall that the mapping $T:\HH\to\HH$ is said to be \emph{uniformly firmly nonexpansive} with modulus $\phi$, $\phi$-uniformly firmly nonexpansive for short, if $T$ is firmly nonexpansive with modulus $\phi$ and $\phi(t)=0 \Leftrightarrow t=0$. For the sake of convenience, we will identify a modulus $\phi:[0,\infty[\to[0,\infty]$ with $\phi(|\cdot|)$, its symmetric extension to $\RR$. With this convention, when we say that $\phi$ is increasing we mean that $0\leq t_1<t_2\ \Rightarrow\ \phi(t_1)\leq\phi(t_2)$.

\begin{proposition}\label{proximal average of positive functions which vanish at zero}
Suppose that for each $i\in I$, $f_i:\HH\to[0,\infty]$ is proper, lower semicontinuous, convex and $f_i(0)=0$. Then $p_\mu({\bf f},{\boldsymbol \lambda}):\HH\to[0,\infty]$ is proper, lower semicontinuous, convex and $p_\mu({\bf f},{\boldsymbol \lambda})(0)=0$. If there exists $i_0\in I$ such that $f_{i_0}(x)=0\Leftrightarrow x=0$, then $p_\mu({\bf f},{\boldsymbol \lambda})(x)=0\Leftrightarrow x=0$.
\end{proposition}

\begin{proof}
Fact~\ref{resolvents of subs and proximal reformulation} guarantees that $p_\mu(\boldsymbol{f},\boldsymbol{\lambda})$ is lower semicontinuous, convex and that for every $x\in\HH$,
\begin{equation}\label{applied proximal reformulation}
p_\mu(\bold{f},\boldsymbol{\lambda})(x)=\inf_{\sum_{i\in I}\lambda_j y_j=x}\sum_{i\in I}\lambda_if_i(y_i)+\frac{1}{\mu}\bigg(\Big(\sum_{i\in I}\lambda_iq(y_i)\Big)-q(x)\bigg).
\end{equation}
Since the bracketed term in~\eqref{applied proximal reformulation} is greater or equal to zero and vanishes when for each $i\in I$, $y_i=x=0$, and since for each $i\in I$, $f_i$ is a function which is greater or equal to zero and which vanishes at zero, formula~\eqref{applied proximal reformulation} implies that $p_\mu(\boldsymbol{f},\boldsymbol{\lambda})$ is greater or equal to zero and vanishes at zero. Finally, for each $i\in I$, we now set $A_i=\partial f_i$. Since for each $i\in I$, 0 is a minimizer of $f_i$, we see that $\bigcap_{i\in I}A_i^{-1}(0)\neq\varnothing$. If $f_{i_0}(x)=0\Leftrightarrow x=0$, then $\{0\}=A^{-1}_{i_0}(0)=\bigcap_{i\in I}A_i^{-1}(0)$. Consequently, equation~\eqref{common inverse solution} together with equation~\eqref{sub of proximal average} imply that $\partial p_\mu(\boldsymbol{f},\boldsymbol{\lambda})^{-1}(0)=\mathcal{R}_\mu(\bold{A},\boldsymbol{\lambda})^{-1}(0)=\{0\}$. We conclude that 0 is the only minimizer of $p_\mu(\boldsymbol{f},\boldsymbol{\lambda})$ and, therefore, $p_\mu(\boldsymbol{f},\boldsymbol{\lambda})(x)=0\Leftrightarrow x=0$.
\end{proof}

\begin{proposition}\label{uniform firmly dominance}
Suppose that for each $i\in I$, $T_i:\HH\to\HH$ is firmly nonexpansive with modulus $\phi_i$ which is lower semicontinuous and convex and set $T=\sum_{i\in I}\lambda_i T_i$. Then $T$ is firmly nonexpansive with modulus $\phi=p_{\frac{1}{2}}(\boldsymbol{\phi},\boldsymbol{\lambda})$ which is proper, lower semicontinuous and convex. In particular, if there exists $i_0\in I$ such that $T_{i_0}$ is $\phi_{i_0}$-uniformly firmly nonexpansive, then $T$ is $\phi$-uniformly firmly nonexpansive.
\end{proposition}

\begin{proof}
Fact~\ref{resolvents of subs and proximal reformulation} guarantees that $\phi=p_{\frac{1}{2}}(\boldsymbol{\phi},\boldsymbol{\lambda})$ is lower semicontinuous, convex and that for every $t\in[0,\infty[$,
\begin{equation}\label{real line proximal average reformulation}
\phi(t)=p_{\frac{1}{2}}(\boldsymbol{\phi},\boldsymbol{\lambda})(t)=\inf_{\sum_{i\in I}\lambda_i t_i=t}\sum_{i\in I}\lambda_i\phi_i(t_i)+2\bigg(\Big(\sum_{i\in I}\lambda_iq(t_i)\Big)-q(t)\bigg).
\end{equation}
Proposition~\ref{proximal average of positive functions which vanish at zero} guarantees that $\phi$ is greater or equal to zero and vanishes at zero. Furthermore, if $\phi_{i_0}$ vanishes only at zero, then $\phi$ vanishes only at zero. Since $\phi$ is convex, we now see that it is increasing. Now, let $x$ and $y$ be points in $\HH$. Then by employing formula~\eqref{real line proximal average reformulation} we obtain the following evaluation
\begin{align*}
\phi\big(\|Tx-Ty\|\big)&\leq\phi\Big(\sum_{i\in I}\lambda_i\|T_ix-T_iy\|\Big)\nonumber\\
&\leq\sum_{i\in I}\lambda_i\phi_i\big(\|T_ix-T_iy\|\big)+\sum_{i\in I}\lambda_i\|T_ix-T_iy\|^2-\Big(\sum_{i\in I}\lambda_i\|T_ix-T_iy\|\Big)^2\nonumber\\
&\leq\sum_{i\in I}\lambda_i\phi_i\big(\|T_ix-T_iy\|\big)+\sum_{i\in I}\lambda_i\|T_ix-T_iy\|^2-\Big\|\sum_{i\in I}\lambda_i(T_ix-T_iy)\Big\|^2
\end{align*}
which implies that
\begin{align*}
\phi\big(\|Tx-Ty\|\big)+\|Tx-Ty\|^2&\leq\sum_{i\in I}\lambda_i\Big(\phi_i\big(\|T_ix-T_iy\|\big)+\|T_ix-T_iy\|^2\Big)\\
&\leq\sum_{i\in I}\lambda_i\scal{T_ix-T_iy}{x-y}=\scal{Tx-Ty}{x-y}.
\end{align*}
Thus, we see that $T$ is firmly nonexpansive with modulus $\phi$. 
\end{proof}

As a consequence, we obtain the following result:

\begin{theorem}\label{uniform monotonicity with 1-coercive modulus is recessive}
\textbf{\emph{(uniform monotonicity with a convex modulus is dominant)}}
Suppose that for each $i\in I$, $A_i:\HH\rightrightarrows\HH$ is maximally monotone with modulus $\phi_i$ which is lower semicontinuous and convex. Then $\average$ is monotone with modulus $\phi=p_{\frac{\mu}{2}}(\boldsymbol{\phi},\boldsymbol{\lambda})$ which is lower semicontinuous and convex. In particular, if there exists $i_0\in I$ such that $A_{i_0}$ is $\phi_{i_0}$-uniformly monotone, then $\average$ is $\phi$-uniformly monotone.
\end{theorem}

\begin{proof}
First we consider the case $\mu=1$. To this end we will employ the fact that the maximally monotone mapping $A:\HH\rightrightarrows\HH$ is monotone with modulus $\phi$ if and only if $J_A$ is firmly nonexpansive with modulus $\phi$. Indeed, employing Minty's parametrization \eqref{Minty's parametrization}, we see that $A$ is uniformly monotone with modulus $\phi$ if and only if for every $x$ and $y$ in $\HH$ we have 
$
\phi\big(\|J_Ax-J_Ay\|\big)\leq\scal{J_Ax-J_Ay}{(x-y)-(J_Ax-J_Ay)}
$
which is precisely the firm nonexpansiveness of $J_A$ with modulus $\phi$. Thus, we see that for each $i\in I$, $J_{A_i}$ is firmly nonexpansive with modulus $\phi_i$. Consequently, Proposition~\ref{uniform firmly dominance} guarantees that $J_{\averageonelambda}=\sum_{i\in I}\lambda_i J_{A_i}$ is firmly nonexpansive with modulus $\phi=p_{\frac{1}{2}}(\boldsymbol{\phi},\boldsymbol{\lambda})$, which, in turn, implies that $\averageonelambda$ is monotone with modulus $\phi$. For an arbitrary $0<\mu$, we employ formulae~\eqref{scaled average} and \eqref{scaled proximal average} as follows: any mapping $A:\HH\rightrightarrows\HH$ is $\phi$-monotone if and only if $\mu A$ is $\mu\phi$-monotone. Thus, since we already have that $\mathcal{R}(\mu\bold{A},\boldsymbol{\lambda})$ is $p_{\frac{1}{2}}(\mu\boldsymbol{\phi},\boldsymbol{\lambda})$-monotone, then $\average=\mu^{-1}\mathcal{R}(\mu\bold{A},\boldsymbol{\lambda})$ is $\phi$-monotone where $\phi=\mu^{-1} p_{\frac{1}{2}}(\mu\boldsymbol{\phi},\boldsymbol{\lambda})=\frac{2}{\mu} p(\frac{\mu}{2}\boldsymbol{\phi},\boldsymbol{\lambda})=p_{\frac{\mu}{2}}(\boldsymbol{\phi},\boldsymbol{\lambda})$. Finally, Proposition~\ref{proximal average of positive functions which vanish at zero} guarantees that if $\phi_{i_0}$ vanishes only at zero, then $\phi$ vanishes only at zero.
\end{proof}

We recall (see~\cite[Section 3.5]{Zal}) that the proper function $f:\HH\to\RX$ is said to be \emph{uniformly convex} if there exists a function $\phi:[0,\infty[\to[0,\infty]$ with the property $\phi(t)=0 \Leftrightarrow t=0$ such that for every two points $x$ and $y$ in $\HH$ and every $\lambda\in\ ]0,1[$,
\begin{equation}\label{uniform convexity}
f\big((1-\lambda)x+\lambda y\big)+\lambda(1-\lambda)\phi\big(\|x-y\|\big)\leq(1-\lambda)f(x)+\lambda f(y).
\end{equation}
The largest possible function $\phi$ satisfying~\eqref{uniform convexity} is called the \emph{gauge of uniform convexity} of $f$ and is defined by
$$
\phi_f(t)=\inf\bigg\{\frac{(1-\lambda)f(x)+\lambda f(y)-f((1-\lambda)x+\lambda y)}{\lambda(1-\lambda)}\ \bigg|\ \lambda\in\ ]0,1[,\ x,y\in\dom f,\ \|x-y\|=t \bigg\}. 
$$
We also recall that the proper and convex function $f:\HH\to\RX$ is said to be \emph{uniformly smooth} if there exists a function $\psi:[0,\infty[\to[0,\infty]$ with the property $\lim_{t\to 0}{\psi(t)}/{t}=0$ such that for every two points $x$ and $y$ in $\HH$ and every $\lambda\in\ ]0,1[$ such that $(1-\lambda)x+\lambda y\in\dom f$,
\begin{equation}\label{uniform smoothness}
f\big((1-\lambda)x+\lambda y\big)+\lambda(1-\lambda)\psi\big(\|x-y\|\big)\geq(1-\lambda)f(x)+\lambda f(y).
\end{equation}
The smallest possible function $\psi$ satisfying~\eqref{uniform smoothness} is called the \emph{gauge of uniform smoothness} of $f$ and is defined by
$$
\psi_f(t)=\sup\bigg\{\frac{(1-\lambda)f(x)+\lambda f(y)-f((1-\lambda)x+\lambda y)}{\lambda(1-\lambda)}\ \bigg|\ \begin{array}{c}
\lambda\in\ ]0,1[,\ x,y\in\HH,\ \|x-y\|=t, \\ 
\big((1-\lambda)x+\lambda y\big)\in\dom f
\end{array}\  \bigg\}. 
$$

\begin{fact}\label{uniform convexity equivalent to uniform monotonicity}\emph{\cite[Theorem 3.5.10, (i)$\Leftrightarrow$(v)]{Zal}}
Suppose that $f:\HH\to\RX$ is proper, lower semicontinuous and convex. Then $f$ is uniformly convex if and only if $\partial f$ is uniformly monotone, in which case $\partial f$ is $\phi$-uniformly monotone with $\phi=2\phi_f^{**}$. 
\end{fact}

Consequently, we arrive at the following results:

\begin{theorem}\label{uniform convexity is dominant or recessive}
\textbf{\emph{(uniform convexity is dominant w.r.t. $p_\mu$)}}
Suppose that for each $i\in I$, $f_i:\HH\to\RX$ is proper, lower semicontinuous, convex and there exists $i_0\in I$ such that $f_{i_0}$ is uniformly convex. Then $p_\mu({\bf f},{\boldsymbol \lambda})$ is uniformly convex.
\end{theorem}

\begin{proof}
For each $i\in I$ we set $A_i=\partial f_i$. Since $f_{i_0}$ is uniformly convex, Fact~\ref{uniform convexity equivalent to uniform monotonicity} guarantees that there exists a lower semicontinuous and convex modulus $\phi_{i_0}$ such that $A_{i_0}$ is $\phi_{i_0}$-uniformly monotone. Consequently, Theorem~\ref{uniform monotonicity with 1-coercive modulus is recessive} together with equation~\eqref{sub of proximal average} guarantee that $\average=\partial p_\mu({\bf f},{\boldsymbol \lambda})$ is uniformly monotone, which, in turn, implies that $p_\mu({\bf f},{\boldsymbol \lambda})$ is uniformly convex.
\end{proof}

\begin{fact}\label{gage conjugation}\emph{\cite[Theorem 3.5.5]{Zal} }
Suppose that $f:\HH\to\RX$ is proper, lower semicontinuous and convex. 
Then \emph{(i)} $f$ is uniformly convex if and only if $f^*$ is uniformly smooth and \emph{(ii)} $f$ is uniformly smooth if and only if $f^*$ is uniformly convex. (Within our reflexive settings \emph{(i)} and \emph{(ii)} are equivalent.) 
\end{fact}

\begin{theorem}\label{uniform smoothness is dominant or recessive}
\textbf{\emph{(uniform smoothness is dominant w.r.t. $p_\mu$)}}
Suppose that for each $i\in I$, $f_i:\HH\to\RX$ is proper, lower semicontinuous, convex and there exists $i_0\in I$ such that $f_{i_0}$ is uniformly smooth. Then $p_\mu({\bf f},{\boldsymbol \lambda})$ is uniformly smooth.
\end{theorem}

\begin{proof} Fact~\ref{gage conjugation} guarantees that $f^*_{i_0}$ is uniformly convex. Consequently, by Theorem~\ref{uniform convexity is dominant or recessive}, $p_{\mu^{-1}}({\bf f^*},{\boldsymbol \lambda})$ is uniformly convex. Finally,  Applying Fact~\ref{gage conjugation} together with formula~\eqref{proximal conjugate} implies that $p_\mu({\bf f},{\boldsymbol \lambda})=p_\mu({\bf f},{\boldsymbol \lambda})^{**}=p_{\mu^{-1}}({\bf f^*},{\boldsymbol \lambda})^*$ is uniformly smooth.
\end{proof}

\begin{remark}
A remark regarding the sharpness of our results is now in order. Under the generality of the hypotheses of Theorem~\ref{uniform monotonicity with 1-coercive modulus is recessive}, even when given that for each $i\in I$, $\phi_i$ is the largest possible modulus of monotonicity of $A_i$, it does not hold that  $\phi=p_{\frac{\mu}{2}}(\boldsymbol{\phi},\boldsymbol{\lambda})$ is necessarily the largest possible modulus of monotonicity of $\average$. In fact, the largest modulus of monotonicity of $\average$ cannot, in general, be expressed only in terms of $\mu,\boldsymbol{\phi}$ and $\boldsymbol{\lambda}$ but depends also on $\bold{A}$. To this end, we consider the following example which illustrates a similar situation for functions and, consequently, for their subdifferential operators. We will illustrate this issue outside the class of subdifferential operators in Example~\ref{strong sharpness counter outside subdeifferentials}  of the following subsection.
\end{remark}

\begin{example}\label{uniform counter example}
Let $f:\HH\to\RX$ be a proper, lower semicontinuous and convex function such that $\phi_f=0$ and $\psi_f=\iota_{\{0\}}$. Among a rich variety of choices, one can choose $f=\iota_C$ where $C$ is a closed half-space. For such a choice of $f$, one easily verifies that, indeed, by their definitions, $\phi_f=0$ and $\psi_f=\iota_{\{0\}}$. Then it follows that $f$ is neither uniformly convex nor uniformly smooth. Consequently, Fact~\ref{gage conjugation} implies that also $f^*$ is neither uniformly convex nor uniformly smooth. However, employing formula~\eqref{proximal average of function with its conjugate}, we see that $p(f,f^*)=q$ which is both, uniformly convex and uniformly smooth (in fact, as we recall in the following subsection, $q$ is strongly convex and strongly smooth). Thus, we see that $\phi_{p_\mu(\bold{f},\boldsymbol{\lambda})}$ and $\psi_{p_\mu(\bold{f},\boldsymbol{\lambda})}$ cannot, in general, be expressed only in terms of $\mu,\boldsymbol{\phi_f},\boldsymbol{\psi_f}$ and $\boldsymbol{\lambda}$ since, for example, in the case of $p(f,f)=f$ we have the same weights $\lambda_1=\lambda_2=\frac{1}{2}$, the same $\mu=1$ and the averaged functions have the same gages of uniform convexity and smoothness which are 0 and $\iota_{\{0\}}$, respectively, as in the case $p(f,f^*)=q$. However, $\phi_{p(f,f)}=0$ and $\psi_{p(f,f)}=\iota_{\{0\}}$, that is, $p(f,f)$ is neither uniformly convex nor uniformly smooth.      
\end{example}

\begin{remark}
Before ending the current discussion a remark regarding the hypothesis of Theorem~\ref{uniform monotonicity with 1-coercive modulus is recessive} is in order. In the general framework of monotone operators, we are not aware of a study of finer properties of the modulus of uniform monotonicity. In particular, existence of a convex modulus (like in the case of uniformly monotone subdifferential operators for which there exists a modulus which is convex and which possesses also other attractive properties (see \cite[Section 3.5]{Zal}), which is crucial in the proof of Theorem~\ref{uniform monotonicity with 1-coercive modulus is recessive}, is unavailable to us. At this point we relegate such a finer study for future research.
\end{remark}

\subsection{Strong monotonicity and cocoercivity}

We say that the mapping $A:\HH\rightrightarrows\HH$ is \emph{$\epsilon$-monotone}, where $\epsilon\geq 0$, if $A-\epsilon\Id$ is monotone, that is, if for any two points $(x,u)$ and $(y,v)$ in $\gr A$, 
$$
\epsilon\|x-y\|^2\leq\scal{v-u}{x-y}.
$$
In particular, a $0$-monotone mapping is simply a monotone mapping. We now recall that $A$ is said to be strongly monotone with constant $\epsilon$, $\epsilon$-strongly monotone for short, if it is $\epsilon$-monotone with $0<\epsilon$. Clearly, $\epsilon$-monotone implies $\epsilon'$-monotone  for any $0\leq\epsilon'<\epsilon$. Letting $\phi(t)=\epsilon t^2$, we see that $A$ is $\epsilon$-monotone if and only if $A$ is monotone with modulus $\phi$. Thus, the subject matter of our current discussion is a particular case of our discussion in the proceeding subsection. However, in view of the importance of strong monotonicity, cocoercivity, strong convexity, strong smoothness and Lipschitzness of the gradient (all which will be defined shortly), we single out these subjects and treat them separately. Moreover, in the present discussion we add quantitative information in terms of explicit constants of the above properties. To this end, we will make use of the following notations and conventions.

First we fix the following conventions in $[0,\infty]$: $0^{-1}=\infty,\ \infty^{-1}=0$ and $0\cdot\infty=0$. For $S\subseteq\HH$ we fix $\infty\cdot S=\HH$ if $0\in S$ and $\infty\cdot S=\varnothing$ if $0\notin S$. We will apply these conventions in the case where $\HH=\RR$ in order to calculate terms of the form $\alpha\Id$ and $\alpha q$ where $\alpha\in[0,\infty]$, however, these calculation hold in any Hilbert space $\HH$. With these conventions at hand, it now follows that $\infty q=\iota_{\{0\}}$, $(\infty q)^*=0$ and $\partial(\infty q)=\partial\iota_{\{0\}}=N_{\{0\}}=\infty\cdot\Id=\infty\partial q$. Consequently, for every $\alpha\in[0,\infty]$ we deduce the following formulae: $(\alpha q)^*=\alpha^{-1}q$, $\partial\alpha q=\alpha\Id=\alpha\partial q$ and $\partial(\alpha q)^*=\alpha^{-1}\Id=(\alpha\Id)^{-1}=(\partial\alpha q)^{-1}$. 

Suppose that for each $i\in I$, $0\leq\alpha_i\leq\infty$ and set $\boldsymbol{\alpha}=(\alpha_1\cdots,\alpha_n)$. Then we define 
\begin{equation}\label{r average}
r_\mu(\boldsymbol{\alpha},\boldsymbol{\lambda})=\big[\sum_{i\in I}\lambda_i(\alpha_i+\mu^{-1})^{-1}\big]^{-1}-\mu^{-1}\ \ \ \text{and}\ \ \ \ r(\boldsymbol{\alpha},\boldsymbol{\lambda})=r_1(\boldsymbol{\alpha},\boldsymbol{\lambda}).
\end{equation}
We note that $0<r_\mu(\boldsymbol{\alpha},\boldsymbol{\lambda})$ if and only if there exists $i_0\in I$ such that $0<\alpha_{i_0}$ and $r_\mu(\boldsymbol{\alpha},\boldsymbol{\lambda})<\infty$ if and only if there exists $i_0\in I$ such that $\alpha_{i_0}<\infty$. For each $i\in I$ we set $f_i=\alpha_i q$ and $A_i=\alpha_i\Id=\partial f_i$. Then by combining our settings and calculations above with either a direct computations or with our formulae from Section 2 (namely, formulae~\eqref{resolvent average of sub},~\eqref{proximal conjugate},~\eqref{niceselfdual} and~\eqref{scaled average}), the following properties of $r_\mu$ follow:  
\begin{align}
&\average=r_\mu(\boldsymbol{\alpha},\boldsymbol{\lambda})\Id,\label{r and R}\\
&p_\mu(\bold{f},\boldsymbol{\lambda})=r_\mu(\boldsymbol{\alpha},\boldsymbol{\lambda})q,\label{r and p}\\
&r_\mu(\boldsymbol{\alpha},\boldsymbol{\lambda})^{-1}=r_{\mu^{-1}}(\boldsymbol{\alpha^{-1}},\boldsymbol{\lambda}),\ \ \ \ \ \text{that is,}\ \ \ \ r_{\mu^{-1}}(\boldsymbol{\alpha^{-1}},\boldsymbol{\lambda})^{-1}=r_\mu(\boldsymbol{\alpha},\boldsymbol{\lambda}),\label{r inverse}\\
&r_{\mu}(\boldsymbol{\alpha},\boldsymbol{\lambda})=\mu^{-1}r(\mu\boldsymbol{\alpha},\boldsymbol{\lambda}).\label{scaled r}
\end{align}
For positive real numbers $\alpha_i$, formula~\eqref{r and p}  was obtained in \cite[formula (25)]{BGLW}.

\begin{theorem}\label{strong monotonicity dominance}\textbf{\emph{(strong monotonicity is dominant)}} Suppose that for each $i\in I$, $\epsilon_i\geq0$, $A_i:\HH\rightrightarrows\HH$ is maximally monotone and $\epsilon_i$-monotone. Then $\average$ is $\epsilon$-monotone where $\epsilon=r_\mu(\boldsymbol{\epsilon},\boldsymbol{\lambda})$. In particular, if there exists $i_0\in I$, such that $A_{i_0}$ is $\epsilon_{i_0}$-strongly monotone, then $\average$ is $\epsilon$-strongly monotone.
\end{theorem}

\begin{proof}
For each $\in I$ we let $\phi_i(t)=\epsilon_i t^2=2\epsilon_i(t^2/2)$. Then $A_i$ is monotone with modulus $\phi_i$. Consequently, Theorem~\ref{uniform monotonicity with 1-coercive modulus is recessive} guarantees that $\average$ is monotone with modulus $\phi=p_\frac{\mu}{2}(\boldsymbol{\phi},\boldsymbol{\lambda})$. By employing formulae~\eqref{r and p} and~\eqref{scaled r} we see that for every $t\geq0$, $\phi(t)=r_\frac{\mu}{2}(2\boldsymbol{\epsilon},\boldsymbol{\lambda})(t^2/2)=r_\mu(\boldsymbol{\epsilon},\boldsymbol{\lambda})t^2$ which means that $\average$ is $\epsilon$-monotone where $\epsilon=r_\mu(\boldsymbol{\epsilon},\boldsymbol{\lambda})$. In particular, if there exists $i_0\in I$, such that $\epsilon_{i_0}>0$, then $r_\mu(\boldsymbol{\epsilon},\boldsymbol{\lambda})>0$.
\end{proof}

We recall that the mapping $A:\HH\rightrightarrows\HH$ is said to
be $\epsilon$-cocoercive, where $\epsilon>0$, if $A^{-1}$ is
$\epsilon$-strongly monotone, that is, if for every pair of points
$(x,u)$ and $(y,v)$ in $\gr A$, $\epsilon\|u-v\|^2\leq\scal{u-v}{x-y}$.
The following result is an immediate consequence of Theorem~\ref{strong
monotonicity dominance}.

\begin{corollary}\label{cocoecivity dominance}\textbf{\emph{(cocoerciveness is dominant)}}
Suppose that for each $i\in I$, $\epsilon_i\geq 0$, $A_i:\HH\rightrightarrows\HH$ is maximally monotone and $A_i^{-1}$ is $\epsilon_i$-monotone. Then $(\average)^{-1}$ is $\epsilon$-monotone where $\epsilon=r_{\mu^{-1}}(\boldsymbol{\epsilon},\boldsymbol{\lambda})$. In particular, if there exists $i_0\in I$ such that $A_{i_0}$ is $\epsilon_{i_0}$-cocoercive, then $\averageonelambda$ is $\epsilon$-cocoercive.
\end{corollary}

\begin{proof}
Since for each $i\in I$, $A_{i}^{-1}$ is $\epsilon_i$-monotone, then Theorem~\ref{strong monotonicity dominance} together with Theorem~\ref{mainresult} guarantee that $(\average)^{-1}=\averageinverse$ is $\epsilon$-monotone. 
\end{proof}

We say that the proper function $f:\HH\to\RX$ is $\epsilon$-convex,
where $\epsilon\geq 0$, if $f$ is convex with modulus
$\phi(t)=\frac{\epsilon}{2}t^2$ (see~\eqref{uniform convexity}).
In particular, a $0$-convex function is simply a convex function.
We recall that $f$ is said to be \emph{strongly convex} with constant
$\epsilon$, $\epsilon$-strongly convex for short, if $f$ is
$\epsilon$-convex and $\epsilon>0$.  We say that the proper and
convex function $f:\HH\to\RX$ is $\epsilon$-smooth, where
$\epsilon\geq 0$, if $f$ is smooth with modulus
$\psi(t)=\frac{1}{2\epsilon}t^2$ (see~\eqref{uniform smoothness}).
In particular, a 0-smooth function is any proper and convex function.
We now recall that $f$ is said to be \emph{strongly smooth} with
constant $\epsilon$, $\epsilon$-strongly smooth for short, if $f$
is $\epsilon$-smooth and $\epsilon>0$.

\begin{fact}\label{Zal strong convexity iff}\emph{\cite[Corollary 3.5.11 and Remark 3.5.3]{Zal}} 
Suppose that $f:\HH\to\RX$ is proper, lower semicontinuous and convex. Then the following assertions are equivalent:
\begin{enumerate}
\item $f$ is $\epsilon$-convex;
\item $\partial f$ is $\epsilon$-monotone;
\item $f^*$ is $\epsilon$-smooth.
\end{enumerate}
If  $\epsilon>0$, assertions \emph{(i), (ii)} and \emph{(iii)} above are equivalent to the following assertion:
\begin{enumerate}\setcounter{enumi}{3}
\item $\dom f^*=\HH$, $f^*$ is Fr\'{e}chet differentiable on $\HH$ and $\partial f^*=\nabla f^*$ is $\frac{1}{\epsilon}$-Lipschitz.
\end{enumerate}
\end{fact}

As a consequence, we obtain the following results:

\begin{theorem}\label{strong convexity dominance}\textbf{\emph{(strong convexity is dominant w.r.t. $p_\mu$)}}
Suppose that for each $i\in I$, $f_i:\HH\to\RX$ is proper, lower semicontinuous and $\epsilon_i$-convex. Then $p_\mu({\bf f},{\boldsymbol \lambda})$ is $\epsilon$-convex where $\epsilon=r_\mu(\boldsymbol{\epsilon},\boldsymbol{\lambda})$. In particular, if there exists $i_0\in I$ such that $f_{i_0}$ is $\epsilon_{i_o}$-strongly convex, then $p_\mu({\bf f},{\boldsymbol \lambda})$ is $\epsilon$-strongly convex.
\end{theorem}

\begin{proof}
For each $i\in I$ we set $A_i=\partial f_i$. Then Fact~\ref{Zal strong convexity iff} guarantees that $A_i$ is $\epsilon_i$-monotone. Consequently, Theorem~\ref{strong monotonicity dominance} together with equation~\eqref{sub of proximal average} imply that $\average=\partial p_\mu({\bf f},{\boldsymbol \lambda})$ is $\epsilon$-monotone which, in turn, implies that $p_\mu({\bf f},{\boldsymbol \lambda})$ is $\epsilon$-convex. In particular, if there exists $i_0\in I$ such that $f_{i_0}$ is $\epsilon_{i_0}$-strongly convex, then $p_\mu({\bf f},{\boldsymbol \lambda})$ is $\epsilon$-strongly convex.
\end{proof}

\begin{theorem}\label{strong smoothness dominance}\textbf{\emph{(strong smoothness is dominant w.r.t. $p_\mu$)}}
Suppose that for each $i\in I$, $f_i:\HH\to\RX$ is proper, lower semicontinuous and $\epsilon_i$-smooth. Then $p_\mu({\bf f},{\boldsymbol \lambda})$ is $\epsilon$-smooth where $\epsilon=r_{\mu^{-1}}(\boldsymbol{\epsilon},\boldsymbol{\lambda})$. In particular, if there exists $i_0\in I$ such that $f_{i_0}$ is $\epsilon_{i_0}$-strongly smooth, then $p_\mu({\bf f},{\boldsymbol \lambda})$ is $\epsilon$-strongly smooth.
\end{theorem}

\begin{proof}
Fact~\ref{Zal strong convexity iff} asserts that each $f_i^*$ is $\epsilon$-convex. Consequently, Theorem~\ref{strong convexity dominance} and formula~\eqref{proximal conjugate} guaranty that $p_{\mu^{-1}}({\bf f^*},{\boldsymbol \lambda})=p_\mu({\bf f},{\boldsymbol \lambda})^*$ is $\epsilon$-convex which, in turn, imply that $p_\mu({\bf f},{\boldsymbol \lambda})=p_\mu({\bf f},{\boldsymbol \lambda})^{**}=p_{\mu^{-1}}({\bf f^*},{\boldsymbol \lambda})^*$ is $\epsilon$-smooth.
\end{proof}

\begin{theorem}\label{Lip grad}\textbf{\emph{(having a Lipschitz gradient is dominant w.r.t. $p_\mu$)}} 
Suppose that for each $i\in I$, $f_i:\HH\to\RX$ is proper, lower semicontinuous, convex and set $A_i=\partial f_i$. Suppose further that there exist $\varnothing\neq I_0\subseteq I$ such that  for every $i\in I_0$, $f_{i}$ is Fr\'{e}chet differentiable on $\HH$, $\nabla f_{i}$ is $\epsilon_{i}$-Lipschitz and for every $i\notin I_0$ set $\epsilon_i=\infty$. Then $p_\mu({\bf f},{\boldsymbol \lambda})$ is Fr\'{e}chet differentiable on $\HH$ and $\average=\nabla p_\mu({\bf f},{\boldsymbol \lambda})$ is $\epsilon$-Lipschitz where $\epsilon=r_\mu(\boldsymbol{\epsilon},\boldsymbol{\lambda})$. 
\end{theorem}

\begin{proof}
Fact~\ref{Zal strong convexity iff} guarantees that for each $i\in I$, $f_i^*$ is $\frac{1}{\epsilon_i}$-convex. By applying Theorem~\ref{strong convexity dominance} we see that $p_{\mu^{-1}}({\bf f^*},{\boldsymbol \lambda})$ is $\frac{1}{\epsilon}$-convex where $\frac{1}{\epsilon}=r_{\mu^{-1}}(\boldsymbol{\epsilon}^{-1},\boldsymbol{\lambda})$. Furthermore, since $I_0\neq\varnothing$, we see that $0<\frac{1}{\epsilon}<\infty$ and, consequently, that $p_{\mu^{-1}}({\bf f^*},{\boldsymbol \lambda})$ is $\frac{1}{\epsilon}$-strongly convex. By applying Fact~\ref{Zal strong convexity iff} together with equation~\eqref{sub of proximal average} and equation~\eqref{proximal conjugate} we conclude that  $p_\mu({\bf f},{\boldsymbol \lambda})=p_{\mu^{-1}}({\bf f}^*,{\boldsymbol \lambda})^*$ is Fr\'{e}chet differentiable on $\HH$ and that $\average=\partial p_\mu({\bf f},{\boldsymbol \lambda})=\nabla p_\mu({\bf f},{\boldsymbol \lambda})$ is $\epsilon$-Lipschitz where, by applying formula~\eqref{r inverse}, $\epsilon=r_{\mu^{-1}}(\boldsymbol{\epsilon}^{-1},\boldsymbol{\lambda})^{-1}=r_\mu(\boldsymbol{\epsilon},\boldsymbol{\lambda})$.
\end{proof}

\begin{remark}
A remark regarding the sharpness of our results is now in order. Although in particular cases, our constants of monotonicity, Lipschitzness and cocoercivity of the resolvent average as well as the constants of convexity and smoothness of the proximal average are sharp, in general, they are not. Furthermore, such sharp constants cannot be determined only by $\mu$, the weight $\boldsymbol{\lambda}$ and the given constant $\boldsymbol{\epsilon}$ even if the latter is sharp, but depend also on the averaged objects as well. We now illustrate these situations by the following examples. 
\end{remark}

\begin{example}\label{strongly firmly example}
For each $i\in I$ let $0<\epsilon_i<\infty$, $f_i=\epsilon_i q$ and $A_i=\epsilon_i\Id=\nabla f_i$. Then for each $i\in I$, $A_i$ is $\epsilon_i$-strongly monotone which is equivalent to $f_i$ being $\epsilon$-strongly convex. As we have seen by formulae~\eqref{r and R} and~\eqref{r and p} we have $\average=r_\mu(\boldsymbol{\epsilon},\boldsymbol{\lambda})\Id$ which is $r_\mu(\boldsymbol{\epsilon},\boldsymbol{\lambda})$-strongly monotone and $p_\mu(\bold{f},\boldsymbol{\lambda})=r_\mu(\boldsymbol{\epsilon},\boldsymbol{\lambda})q$ which is $r_\mu(\boldsymbol{\epsilon},\boldsymbol{\lambda})$-strongly convex. Furthermore, for any $\epsilon'>r_\mu(\boldsymbol{\epsilon},\boldsymbol{\lambda})$ we see that $\average-\epsilon'\Id$ is not monotone, that is, $\average$ is not $\epsilon'$-monotone, and that $p_\mu(\bold{f},\boldsymbol{\lambda})-\epsilon' q$ is not convex, that is, $p_\mu(\bold{f},\boldsymbol{\lambda})$ is not $\epsilon'$-convex. Thus, we conclude that $\epsilon=r_\mu(\boldsymbol{\epsilon},\boldsymbol{\lambda})$ is a sharp constant of strong monotonicity of the resolvent average and as a constant of strong convexity of the proximal average in this example.
\end{example}

\begin{example}\label{strong sharpness counter for subdifferentials}
We consider the settings in Example~\ref{uniform counter example}. We see that neither $f$ nor $f^*$ is strongly convex nor strongly smooth. Also, neither $\partial f$ nor $\partial f^*$ is strongly monotone or Lipschitz continuous. However, $p(f,f^*)=q$ is 1-strongly convex and 1-strongly smooth and $\mathcal{R}(\partial f,\partial f^*)=\Id$ is 1-strongly monotone and 1-Lipschitz. Letting $\epsilon_1=\epsilon_2=0$ (which is the only possible constant of monotonicity of $\partial f$ and of $\partial f^*$, and the only possible constant of convexity and smoothness of $f$ and $f^*$), we see that $\epsilon=r(\boldsymbol{\epsilon},\boldsymbol{\lambda})=0$. Thus, we conclude that $\epsilon$ is not sharp in this case. On the other hand, $\epsilon=0$ will be the sharp constant of convexity when we consider $p(0,0)=0$ and a sharp constant of monotonicity and Lipschitzness of $\mathcal{R}(0,0)=0$ where the given $\mu$, $\boldsymbol{\lambda}$ and $\boldsymbol{\epsilon}$ are the same. Thus we conclude that the sharp constants of monotonicity, Lipschitzness and cocoercivity of the resolvent average as well as the sharp constants of convexity and smoothness of the proximal average cannot, in general, be determined only by the corresponding constants of the averaged objects, the parameter $\mu$ and the weight $\boldsymbol{\lambda}$. 
\end{example}

\begin{example}\label{strong sharpness counter outside subdeifferentials} 
Outside the class of subdifferential operators, we consider the following settings: let $\lambda_1=\lambda_2=\frac{1}{2}$, let $A_1$ be the counterclockwise rotation in the plane by the angle $0<\alpha\leq\frac{\pi}{2}$ and let $A_2=A_1^{-1}$ be the clockwise rotation by the angle $\alpha$. Then $A_1$ and $A_2$ are $\epsilon$-monotone where $\epsilon=\cos\alpha<1$ is sharp. In particular, for $\alpha=\frac{\pi}{2}$, $A_1$ and $A_2$ are not strongly monotone. However, by employing formula~\eqref{self dual is I}, we see that $\mathcal{R}(A_1,A_2)=\Id$ which is 1-strongly monotone. On the other hand, $\mathcal{R}(A_1,A_1)=A_1$ is $\epsilon$-monotone where, again, $\epsilon=\cos\alpha$ is sharp. Thus, this is another demonstration to the fact that the sharp constant of monotonicity of the resolvent average does not depend only on the constants of monotonicity of the averaged mappings, the weight $\boldsymbol{\lambda}$ and the parameter $\mu$. We will analyze a variant of this example in order to discuss Lipschitzness outside the class of subdifferential operators in Example~\ref{scaled rotations}. 
\end{example}

\subsection{Disjoint injectivity}
We recall that the mapping $A:\HH\rightrightarrows\HH$ is said to be \emph{disjointly injective} if for any two distinct point $x$ and $y$ in $\HH$, $Ax\cap Ay=\varnothing$.

\begin{theorem}\label{t:disjinjdom}\textbf{\emph{(disjoint injectivity is dominant)}}
Suppose that for each $i\in I$, $A_i:\HH\rightrightarrows\HH$ is maximally monotone. If there exists $i_0\in I$ such that $A_{i_0}$ is disjointly injective, then $\average$ is disjointly injective.
\end{theorem}

\begin{proof}
We recall that the maximally monotone mapping $A:\HH\rightrightarrows\HH$ is disjointly injective if and only if $J_A$ is strictly nonexpansive, that is, for any two distinct points $x$ and $y$ in $\HH$, $\|J_A x-J_A y\|<\|x-y\|$ (see \cite[Theorem 2.1(ix)]{bmw12}). We also note that $A$ is disjointly injective if and only if $\mu A$ is disjointly injective. Thus, since for every $i\in I$, $J_{\mu A_i}$ is nonexpansive and since $J_{\mu A_{i_0}}$ is strictly nonexpansive, then the convex combination $J_{\mu\average}=\sum_{i\in I}\lambda_i J_{\mu A_i}$ is strictly nonexpansive, which, in turn, implies that $\average$ is disjointly injective.
\end{proof}

\begin{remark}
Let $\HH$ be finite-dimensional and let $f:\HH\to\RX$ be proper, lower semicontinuous and convex function. Then the disjoint injectivity of $\partial f$ is equivalent to the essential strict convexity of $f$ which is equivalent to the essential smoothness of $f^*$ (see \cite[Theorem 26.3]{Rock}). Thus, in the finite-dimensional case, once again, we recover Corollary~\ref{ess smooth}, Corollary~\ref{ess conv} and Corollary~\ref{Leg}, that is, essential smoothness, essential strict convexity and Legendreness are dominant properties w.r.t. the proximal average.  
\end{remark}

\section{Recessive properties of the resolvent average}\label{recessive}

We begin this section with the following example of a recessive property:

\begin{example}
\textbf{(being a constant mapping is recessive)}
Suppose that for each $i\in I$, $A_i:\HH\rightrightarrows\HH$ is the constant mapping $x\mapsto z_i$. Then $\average$ is the constant mapping $x\mapsto\sum_{i\in I}\lambda_i z_i$. Indeed, the mapping $A:\HH\rightrightarrows\HH$ is the constant mapping $x\mapsto z$ if and only if $J_A$ is the shift mapping $x\mapsto x-z$. Thus, we see that $J_{\mu\average}=\sum_{i\in I}\lambda_i J_{\mu A_i}$ is the shift mapping $x\mapsto x-\mu\sum_{i\in I}\lambda_i z_i$ and, consequently, that $\average$ is the constant mapping $x\mapsto\sum_{i\in I}\lambda_i z_i$. However, if we let $A_1$ be any constant mapping, $A_2$ be any maximally monotone mapping which is not constant, $0<\lambda<1,\ \lambda_1=\lambda$ and $\lambda_2=1-\lambda$, then $J_{\mu A_1}$ is a shift, $J_{\mu A_2}$ is not a shift and  $J_{\mu\average}=\lambda_1 J_{\mu A_1}+(1-\lambda_1)J_{\mu A_2}$ is not a shift. Consequently, $\average$ is not a constant mapping. Summing up, we see that being a constant mapping is recessive w.r.t. the resolvent average.
\end{example}

\subsection{Linearity and affinity}
We begin our discussion with the following example:

\begin{example}
\label{linearcounter}
We set $f=\| \cdot\|$, $A_1 = \partial f$ and $A_2 = {\bf 0}$. Then
$$J_{A_1} x = \begin{cases}\begin{matrix} \left(1-\frac{1}{\|x\|}\right)x, & \text{if } \|x\|>1;
\\ 0, & \text{if } \|x\| \leq 1
\end{matrix}
\end{cases}$$
and $J_{A_2} = \Id$ (see \cite[Example 23.3 and Example 14.5]{BC2011}). We see that $J_{A_1}$ is not an affine relation and $J_{A_2}$ is linear. However, letting $0<\lambda<1,\ \lambda_1=\lambda$ and $\lambda_2=1-\lambda$, then
$$J_{\averageonelambda}x = \lambda J_{A_1}x + (1-\lambda) J_{A_2}x =\begin{cases}\begin{matrix} \left(1-\lambda \frac{1}{\|x\|}\right)x, & \text{if } \|x\|>1;
\\ (1-\lambda)x, & \text{if } \|x\| \leq 1,
\end{matrix}
\end{cases}$$
which is not an affine relation. Thus, by employing Fact~\ref{linear}, we conclude that $\averageonelambda$ is not an affine relation.
\end{example}

Example~\ref{linearcounter} demonstrates that linearity and affinity are not dominant properties w.r.t. the resolvent average. On the other hand, since the sums and inversions in the definition of $\average$ preserve linearity and affinity, we arrive at the following observation:  

\begin{corollary}[Linearity and affinity are recessive]\label{linearrelation}
Suppose that for each $i\in I$, $A_i:\HH\rightrightarrows\HH$ is a maximally monotone linear (resp. affine) relation. Then $\average$ is a maximally monotone linear (resp. affine) relation.
\end{corollary}

\subsection{Rectangularity and paramonotonicity}

We recall Definition~\ref{D:para rec def} of rectangular and paramonotone mappings and consider the following example: 
\begin{example}\label{ParaRecDominant}
In $\RR^2$, let $A_{1}=N_{\RR\times \{0\}}.$ Then by Fact~\ref{NCresolvent}, $J_{A_1}$ is the projection on $\RR\times \{0\}$. Since $A_1$ is a subdifferential of a proper, lower semicontinuous and convex function, it is rectangular (see \cite[Example 24.9]{BC2011}) and paramonotone (see \cite[Example 22.3]{BC2011}). Let $A_{2}:\RR^2\rightarrow\RR^2$ be the counterclockwise rotation by $\pi/2$. Then by employing standard matrix representation we write
$$J_{A_{1}}=P_{\RR\times \{0\}}=\begin{pmatrix}
1 & 0\\
0 & 0
\end{pmatrix},\   
A_{2}=\begin{pmatrix}
0 & -1\\
1 & 0
\end{pmatrix},\ 
A_{2+}=\frac{1}{2}(A_2+A_2^{\intercal})=0
\  \text{ and }\ \  J_{A_{2}}=\begin{pmatrix}
\tfrac{1}{2} & \tfrac{1}{2}\\[+0.8mm]
-\tfrac{1}{2} & \tfrac{1}{2}
\end{pmatrix}.
$$
Letting $\lambda_1=\lambda_2=\tfrac{1}{2}$, we obtain
$$
\mathcal{R}({\bf{A}})  =(\frac{1}{2} J_{A_{1}}+\frac{1}{2}J_{A_{2}})^{-1}-\Id
= \begin{pmatrix}
0 & -1\\
1 & 2
\end{pmatrix}
\ \ \ \text{and}\ \ \ 
\mathcal{R}({\bf{A}})_+=\frac{1}{2} \big(\mathcal{R}(\bf{A})+\mathcal{R}(\bf{A})^{\intercal}\big)=\begin{pmatrix}
0 & 0\\
0 & 2
\end{pmatrix}.
$$
By employing Fact~\ref{f:pararecsame} we see that $A_2$ as well as $\averageonelambda$ are neither rectangular nor paramonotone. 
\end{example}

In view of Example~\ref{ParaRecDominant} it is clear that rectangularity and paramonotonicity are not dominant properties w.r.t. the resolvent average. We now prove the recessive nature of these properties. We begin with rectangularity.

\begin{theorem}[rectangularity is recessive]
Suppose that for each $i\in I$, $A_i:\HH\rightrightarrows\HH$ is rectangular and maximally monotone. Then $\average$ is rectangular.
\end{theorem}

\begin{proof}
The mapping $A:\HH\rightrightarrows\HH$ is rectangular if and only if $\mu A$ is rectangular as can be seen from~\eqref{rectangular definition}. Thus, by employing~\eqref{Fitzpatrick rectangular} we see that $\average$ is rectangular as soon as $\dom\mu\average\times\ran\mu\average\subseteq\dom F_{\mu\average}$. Indeed, since for each $i\in I$, $\dom \mu A_i\times\ran\mu A_i\subseteq \dom F_{\mu A_i}$, by employing inclusion~\eqref{resolvents graph} and then inclusion~\eqref{resolvent average Fitzpatrick function domain} we arrive at
\begin{align*}
\dom\average\times\ran\average &\subseteq \sum_{i\in
I}\lambda_i(\dom\mu A_i\times\ran\mu A)\subseteq \sum_{i\in I}\lambda_i\dom F_{\mu A_i}\\
&\subseteq \dom F_{\average},
\end{align*}
as asserted.
\end{proof}

In order to prove that paramonotonicity is recessive, we will make use of the following result:

\begin{proposition}\label{convex combination of resolvents of paramonotone}
Suppose that for each $i\in I,\ T_i:\HH\to\HH$ is firmly nonexpansive and set $T=\sum_{i\in I}\lambda_i T_i$. Then:
\begin{enumerate}
\item If for each $i\in I$, given points $x$ and $y$ in $\HH$,
\begin{equation}\label{resolvent of paramonotone}
\|T_i x-T_i y\|^2=\scal{x-y}{T_i x-T_iy}\ \ \ \ \Rightarrow\ \ \ \ \begin{cases}
T_ix=T_i(T_ix+y-T_i y)\\
T_i y=T_i(T_i y+x-T_i x),
\end{cases}
\end{equation}
then $T$ also has property~\eqref{resolvent of paramonotone}.

\item\label{resolvent of paramonotone pluse injective} If there exists $i_0\in I$ such that $T_{i_0}$  has property~\eqref{resolvent of paramonotone} and is injective, then $T$ has property~\eqref{resolvent of paramonotone} and is injective.
\end{enumerate}
\end{proposition}

\begin{proof}
(i) Suppose that $x$ and $y$ are points in $\HH$ such that $\|Tx-Ty\|^2=\scal{x-y}{Tx-Ty}$. Then Corollary~\ref{equality in convex combination of firmly nonexpansive} guarantees that $T_i x-T_i y=Tx-Ty$ for every $i\in I$. Employing property~\eqref{resolvent of paramonotone} of the $T_i$'s, we obtain the set of equalities in~\eqref{resolvent of paramonotone}. Consequently, we see that
$$
Tx=\sum_{i\in I}\lambda_i T_i x=\sum_{i\in I}\lambda_i T_i(T_ix+y-T_i y)=\sum_{i\in I}\lambda_i T_i(Tx+y-Ty)=T(Tx+y-Ty)
$$
and, similarly, that $Ty=T(Ty+x-Tx)$.
(ii) In the same manner as in the proof of (i), if $\|Tx-Ty\|^2=\scal{x-y}{Tx-Ty}$, then $T_i x-T_i y=Tx-Ty=T_{i_0}x-T_{i_0}y$ for every $i\in I$. Since $T_{i_0}$ is injective and has property~\eqref{resolvent of paramonotone}, then, in fact, $x-y=T_{i_0}x-T_{i_0}y$. Thus, we conclude that $x=Tx+y-Ty$, equivalently $y=Ty+x-Tx$, which confirms that $T$ has property~\eqref{resolvent of paramonotone}. Finally, Proposition~\ref{equality in convex combination of firmly nonexpansive}\ref{firmly injective} guarantees that $T$ is also injective.
\end{proof}

\begin{theorem}\label{paramonotonicity is recessive}\textbf{\emph{(paramonotonicity is recessive)}}
Suppose that for each $i\in I$, $A_i:\HH\rightrightarrows\HH$ is maximally monotone and paramonotone. Then $\average$ is paramonotone.
\end{theorem}

\begin{proof}
First we consider the case $\mu=1$. We recall that the maximally monotone mapping $A:\HH\rightrightarrows\HH$ is paramonotone if and only if $J_A$ has property~\eqref{resolvent of paramonotone} (see \cite[Theorem 2.1(xv)]{bmw12}). Thus, since each $J_{A_i}$ has property~\eqref{resolvent of paramonotone}, then Proposition~\ref{convex combination of resolvents of paramonotone}(i) guarantees that $J_{\averageonelambda}=\sum_{i\in I}\lambda_i J_{A_i}$ also has property~\eqref{resolvent of paramonotone}, which, in turn, implies that $\averageonelambda$ is paramonotone. For arbitrary $0<\mu$, we note that the mapping $A:\HH\rightrightarrows\HH$ is paramonotone if and only if $\mu A$ is paramonotone. Since for each $i\in I$, $\mu A_i$ is paramonotone, so is $\mathcal{R}\bold(\mu A,\boldsymbol{\lambda})$. Consequently, formula~\eqref{scaled average} implies that $\mu^{-1}\mathcal{R}(\bold\mu A,\boldsymbol{\lambda})=\average$ is paramonotone.
\end{proof}

\subsection{$k$-cyclic monotonicity and cyclic monotonicity}
We recall that the mapping $A:\HH\rightrightarrows\HH$ is said to be \emph{$k$-cyclically monotone}, where $k\in\{2,3,\dots\}$, if any $k$ pairs $(x_1,u_1),\dots,(x_k,u_k)\in\gr A$, letting $x_{k+1}=x_1$, satisfies $0\leq\sum_{i=1}^k\scal{u_i}{x_i-x_{i+1}}$. The mapping $A$ is said to be cyclically monotone if it is $k$-cyclically monotone for every $k\in\{2,3,\dots\}$. Rockafellar's well known characterization from \cite{Rock cyc} asserts that proper, maximally monotone and cyclically monotone mappings are precisely the subdifferentials of proper, convex and lower semicontinuous functions. 

\begin{example}
\label{ex:cyclicmonocounter}
Let $\HH = \RR^2$. We set $n=2,\ {\boldsymbol \lambda}=(1/3,2/3)$, we let $A_1$ be the identity and we let $A_2$ be the mapping obtained by the counter-clockwise rotation by $\pi/2$. Then $A_1$ is cyclically monotone while $A_2$ is monotone but not $3$-cyclically monotone (see \cite[Example 4.6]{BBBRW} for further discussion of $k$-cyclic monotonicity of rotations by $\pi/k$). Then, by employing matrix representation, we obtain
$$
\averageonelambda = \frac{1}{13}\begin{pmatrix}5 & -12\\12 & 5\end{pmatrix}.
$$
We now let $x_1=x_4=0$, $x_2=e_1$ and $x_3=e_2$ where $(e_1,e_2)$ is the standard basis of $\RR^2$. Then
$$
\sum_{i=1}^3\scal{\averageonelambda x_i}{x_i-x_{i+1}}=\scal{\averageonelambda e_1}{e_1-e_2}+\scal{\averageonelambda e_2}{e_2}=-2/13<0.
$$
Thus, we see that $\averageonelambda$ is not $3$-cyclically monotone.
\end{example}

In view of Example~\ref{ex:cyclicmonocounter}, the property of $k$-cyclic monotonicity is not dominant w.r.t. the resolvent average. In order to conclude the recessive nature of $k$-cyclic monotonicity and of cyclic monotonicity we recall the following: By employing cyclic monotonicity, the convexity of the set of proximal mappings was recovered in \cite[Theorem 6.7]{BBBRW}. In other words, if we suppose that for each $i\in I$, $A_i:\HH\rightrightarrows\HH$ is maximally monotone and cyclically monotone, then $\averageonelambda$ is maximally monotone and cyclically monotone. Since the proof of this result was actually carried out for every fixed $k$, it, in fact, holds also for $k$-cyclically monotone mappings. Summing up, we arrive at the following conclusion:

\begin{theorem}
\textbf{\emph{($k$-cyclic monotonicity and cyclic monotonicity are recessive)}}
Suppose that for each $i\in I$, $A_i:\HH\rightrightarrows\HH$ is maximally monotone. If for every $i\in I$, $A_i$ is $k$-cyclically monotone, then $\average$ is $k$-cyclically monotone. If for every $i\in I$, $A_i$ is cyclically monotone, then $\average$ is cyclically monotone.
\end{theorem}

\begin{proof}
The case where $\mu=1$ follows from our discussion above. For arbitrary $\mu>0$, we note that the mapping $A:\HH\rightrightarrows\HH$ is $k$-cyclically monotone (cyclically monotone) if and only if the mapping $\mu A$ is $k$-cyclically monotone (cyclically monotone). Thus, we see that $\mathcal{R}(\mu\bf{A},\boldsymbol{\lambda})$ is $k$-cyclically monotone (cyclically monotone), and, consequently, by employing formula~\eqref{scaled average}, so is $\mu^{-1}\mathcal{R}(\mu\bf A,\boldsymbol{\lambda})=\average$.
\end{proof}

\subsection{Weak sequential closedness of the graph}
The mapping $T:\HH\to\HH$ is said to be \emph{weakly sequentially continuous} if it maps a weakly convergent sequence to a weakly convergent sequence, that is, $x_n\rightharpoonup x\ \Rightarrow\ Tx_n\rightharpoonup Tx$. We recall that the maximally monotone mapping $A:\HH\rightrightarrows\HH$ has a weakly sequentially closed graph if and only if $J_A$ is weakly sequentially continuous (see \cite[Theorem 2.1(xxi)]{bmw12}). 

\begin{example}
Within the settings of Example~\ref{linearcounter}, we consider the case where $\HH=\ell^2(\NN)$ and $\{e_1,e_2,\dots\}$ is its standard basis. Then, for every $n\in\NN$, we set $x_n=e_1+e_n$. Consequently, $x_n\rightharpoonup e_1$ and $J_{A_1}x_n=(1-1/\sqrt{2})x_n\rightharpoonup(1-1/\sqrt{2})e_1$ while $J_{A_1}e_1=0$. Thus, $J_{A_1}$ is not weakly sequentially continuous and $J_{A_2}$ is. Furthermore, 
\begin{align*}
J_{\averageonelambda}x_n&=\lambda_1 J_{A_1}x_n+\lambda_2 J_{A_2}x_n=\lambda_1(1-1/\sqrt{2})x_n+\lambda_2 x_n\\
 &\rightharpoonup \lambda_1(1-1/\sqrt{2})e_1+\lambda_2e_1=(1-\lambda_1/\sqrt{2})e_1.
\end{align*}
Since $J_{\averageonelambda}e_1=\lambda_2e_1$, we see that $J_{\averageonelambda}$ is not weakly sequentially continuous, which, in turn, implies that the graph of $\averageonelambda$ is not weakly sequentially closed, although, the graph of $A_2$ is.
\end{example}

We see that the weak sequential closedness of the graph is not a dominant property w.r.t. the resolvent average. 

\begin{corollary}
\textbf{\emph{(weak sequential closedness of the graph is recessive)}} Suppose that for each $i\in I$, $A_i:\HH\rightrightarrows\HH$ is maximally monotone with a weakly sequentially closed graph. Then $\average$  has a weakly sequentially closed graph.
\end{corollary}

\begin{proof}
Since for each $i\in I$, $J_{\mu A_i}$ is weakly sequentially continuous, so is $J_{\mu\average}=\sum_{i\in }\lambda_i J_{\mu A_i}$. Consequently, $\averageonelambda$ has a weakly sequentially closed graph.
\end{proof}

\subsection{Displacement mappings}
The mapping $D:\HH\to\HH$ is said to be a \emph{displacement mapping} if there exists a nonexpansive mapping $N:\HH\to\HH$ such that $D=\Id-N$, in which case $D$ is maximally monotone.

\begin{example}\label{displacement not dominant}
Let $n=2,\ \lambda_1=\lambda_2=\frac{1}{2},\ 0<\alpha_1=\alpha,\ 0<\alpha_2=\beta,\ A_1=\alpha\Id$ and $A_2=\beta\Id$. Then by applying formula~\eqref{r and R} we obtain 
$$
\averageonelambda=r(\boldsymbol{\alpha},\boldsymbol{\lambda})\Id=\frac{\alpha+\beta+2\alpha\beta}{2+\alpha+\beta}\Id.
$$
Now we let $\alpha=2$ and $\beta=5$. Then $A_1$ is a displacement mapping, $A_2$ is not a displacement mapping and $\averageonelambda=3\Id$ is not a displacement mapping.
\end{example}

In view of Example~\ref{displacement not dominant}, we see that being a displacement mapping is not a dominant property w.r.t. the resolvent average. In order to prove that being a displacement mapping is recessive we will make use of the following result:

\begin{proposition}\label{displacement characterization}
The maximally monotone mapping $A:\HH\rightrightarrows\HH$ is $\frac{1}{2}$-strongly monotone if and only if $A^{-1}$ is a displacement mapping. 
\end{proposition}

\begin{proof}
The assertion that the mapping $A$ is $\frac{1}{2}$-strongly monotone is equivalent to $A=\frac{1}{2}\Id+M_1$ for a maximally monotone mapping $M_1$, which, in turn, is equivalent to $A=\frac{1}{2}\Id+M_2\circ(\frac{1}{2}\Id)$ for a maximally monotone mapping $M_2$ (let $M_2=M_1\circ(2\Id$)). By employing equation~\eqref{mu resolvent identity 2}, this is equivalent to $A^{-1}=2[\Id-(M^{-1}_2+\Id)^{-1}]$. Finally, this is equivalent to $A^{-1}=2(\Id-F)$ for a firmly nonexpansive mapping $F$ (let $F=J_{M_2^{-1}}$), which, in turn, is equivalent to $A^{-1}=2(\Id-(\Id+N)/2)=\Id-N$ for a nonexpansive mapping $N$ (let $N=2F-\Id$). 
\end{proof} 

\begin{corollary}\label{monotone and nonexpansive char}
The maximally monotone mapping $N:\HH\to\HH$ is nonexpansive if and only if $N=2J_B-\Id$ for a maximally monotone and nonexpansive mapping $B$.   
\end{corollary}

\begin{proof}
Since $\frac{1}{2}(N+\Id)$ is $\frac{1}{2}$-strongly monotone, the assertion $(B+\Id)^{-1}=J_B=\frac{1}{2}(N+\Id)$ is equivalent to $B+\Id$ being a displacement mapping, that is $\Id+B=\Id-N'$ for a nonexpansive mapping $N'$, that is, $B$ is a nonexpansive mapping.
\end{proof}

\begin{corollary}\label{displacement recessiveness}
\textbf{\emph{(being a displacement mapping is recessive)}}
Suppose that for each $i\in I$, $A_i:\HH\rightrightarrows\HH$ is a displacement mapping. Then $\averageonelambda$ is a displacement mapping.
\end{corollary}

\begin{proof}
By employing Proposition~\ref{displacement characterization} we see that for each $i\in I$, $A_i^{-1}$ is $(1/2)$-strongly monotone. Consequently, Theorem~\ref{strong monotonicity dominance} guarantees that $(\averageonelambda)^{-1}=\averageoneinverse$ is $(1/2)$-strongly monotone, which, in turn, implies that $\averageonelambda$ is a displacement mapping.
\end{proof}

\subsection{Nonexpansive monotone operators}
In Example~\ref{displacement not dominant}, we let $\alpha=1$ and $\beta=5$. Then $A_1=\Id$ is nonexpansive, $A_2=5\Id$ is not nonexpansive and $\averageonelambda=2\Id$ which is not nonexpansive. Thus, we see that nonexpansiveness is not a dominant property. 

\begin{theorem}\label{nonexpansive average}
\textbf{\emph{(nonexpansiveness is recessive)}}
Suppose that for each $i\in I$, $A_i:\HH\to\HH$ is a nonexpansive and monotone mapping. Then $\averageonelambda$ is nonexpansive. Furthermore, for each $i\in I$, $A_i=2J_{B_i}-\Id$ where $B_i$ is maximally monotone, nonexpansive and $\averageonelambda=2J_B-\Id$ where $B$ is the maximally monotone and nonexpansive mapping given by $B=\sum_{i\in I}\lambda_iB_i$. 
\end{theorem}

\begin{proof}
Since the $A_i$'s have full domain and are continuous, they are maximally monotone. Employing Corollary~\ref{monotone and nonexpansive char}, we see that the $B_i$'s are maximally monotone and nonexpansive (and have full domain). Furthermore, we have
\begin{align*}
J_{\averageonelambda}&=\sum_{i\in I}\lambda_iJ_{A_i}=\sum_{i\in
I}\lambda_i(2J_{B_i})^{-1}=\sum_{i\in
I}\lambda_i(J_{B_i})^{-1}\circ(\tfrac{1}{2}\Id)=\sum_{i\in
I}\lambda_i(B_i+\Id)\circ(\tfrac{1}{2}\Id)\\
&=(B+\Id)\circ(\tfrac{1}{2}\Id)=(2J_B)^{-1}.
\end{align*}
Thus, we see that $\averageonelambda=2J_B-\Id$, as asserted.
\end{proof}

\section{Miscellaneous observations and remarks}\label{neither}

In our last section, we consider combinations of properties, indeterminate properties and other observations and remarks.

\subsection{Paramonotonicity combined with single-valudeness}
\begin{theorem}\label{singlevalued paramonotone}
\textbf{\emph{(paramonotonicity combined with single-valudeness is dominant)}} Suppose that for each $i\in I$, $A_i:\HH\rightrightarrows\HH$ is maximally monotone. If for some $i_0\in I$, $A_{i_0}$ is paramonotone and at most single-valued, then so is $\average$.
\end{theorem}
 
\begin{proof}
We recall that the maximally monotone mapping $A:\HH\rightrightarrows\HH$ is at most single-valued if and only if $J_A$ is injective (see \cite[Theorem 2.1(iv)]{bmw12}). We also recall that $A$ is paramonotone if and only if $J_A$ has property~\eqref{resolvent of paramonotone}. Thus, we see that $J_{\mu A_{i_0}}$ is injective and has property~\eqref{resolvent of paramonotone}. Consequently, Proposition~\ref{convex combination of resolvents of paramonotone}\ref{resolvent of paramonotone pluse injective} implies that $J_{\mu\average}=\sum_{i\in I}\lambda_iJ_{\mu A_i}$ is injective and has property~\eqref{resolvent of paramonotone}, which, in turn, implies that $\average$ is paramonotone and at most single-valued.
\end{proof}

\subsection{Linear relations, bounded linear operators and linear operators on $\RR^n$}
Within the class of maximally monotone and linear relation on $\HH$, which we will denote by $MLR(\HH)$, lies the class of classical (single-valued) monotone and bounded linear operators which we will denote by $BML(\HH)$. Within $BML(\HH)$ we will denote by $BMLI(\HH)$ the class of invertible operators, that is, bounded linear monotone surjective operators with a bounded inverse. We begin our discussion of these classes of operators with the following result:

\begin{theorem}\label{classical linearity dominance within}
\textbf{\emph{(in $MLR(\HH)$, being $BML(\HH)$ and being $BMLI(\HH)$ are dominant)}}
Suppose that for each $i\in I$, $A\in MLR(\HH)$ and there exists $i_0\in I$ such that $A_{i_0}\in BML(\HH)$. Then $\average\in BML(\HH)$. Furthermore: 
\begin{enumerate}
\item If $A_{i_0}\in BMLI(\HH)$, then $\average\in BMLI(\HH)$;
\item If $A_{i_0}$ is paramonotone, then $\average$ is paramonotone.
\end{enumerate}
\end{theorem}

\begin{proof}
Corollary~\ref{linearrelation} guarantees that $\average$ is a monotone linear relation. Since $A_{i_0}$ has a full domain, then Theorem~\ref{t:FDdom}\ref{t:Fdom} guarantees that $\average$ has full domain.  Since $A_{i_0}$ is single-valued, then Theorem~\ref{single valued} guarantees that $\average$ is single-valued. Finally, since a monotone operator is locally bounded on the interior of its domain (see  \cite[Theorem 1]{Rock bndd}) and since $\average$ has full domain, then it is everywhere locally bounded, which, in the case of single-valued linear operators, is equivalent to boundedness.  If $A_{i_0}$ is a bounded linear operator which is also invertible, then $A_{i_0}^{-1}$ is bounded, which implies that $\averageinverse=(\average)^{-1}$ is bounded. Finally, if $A_{i_0}$ is paramonotone, then Theorem~\ref{singlevalued paramonotone} guarantees that $\average$ is paramonotone. 
\end{proof}

\begin{remark}
Within the settings and hypothesis of Theorem~\ref{classical linearity dominance within}, we note that when $\HH=\RR^n$, paramonotonicity is interchangeable with rectangularity since, in this case, the two properties are equivalent (see Fact~\ref{f:pararecsame}). 
\end{remark}

We now discuss the class of maximally monotone linear relations on $\RR^n$, which we denote by $MLR(n)$, its subclass of monotone linear mappings (classical, single-valued), which we denote by $ML(n)$, and its subclass of monotone, linear and (classically) invertible operators which we denote by $MLI(n)$. We will identify operators in $ML(n)$ with their standard matrix representation. A subclass of $ML(n)$ is the class of positive semidefinite matrices and a subclass of  $MLI(n)$ is the class of positive definite matrices. The resolvent average of positive semidefinite and definite matrices was studied in \cite{bmw-res}. We now recall that any positive semidefinite matrix $A$ has the property that $0\leq\det A$. This fact holds in the larger class $ML(n)$ as well:
\begin{proposition}\label{det}
Suppose that $A\in ML(n)$. Then for every $0\leq\lambda<1$, $0<\det(\lambda A+(1-\lambda)\Id)$. Consequently, $0\leq\det A$.
\end{proposition} 

\begin{proof}
Let $x\in\RR^n,\ x\neq 0$. Since $0\leq\scal{x}{Ax}$, then 
$$
0<\lambda\scal{x}{Ax}+(1-\lambda)\|x\|^2=\scal{x}{(\lambda A+(1-\lambda)\Id)x}.
$$
We conclude that $0\neq(\lambda A+(1-\lambda)\Id)x$ for every $x\in\RR^n,\ x\neq0 $. We now define $\phi:[0,1]\to\RR$ by $\phi(\lambda)=\det(\lambda A+(1-\lambda)\Id)$. It follows that $\phi$ is continuous, $\phi$ does not vanish on $[0,1[$ and $\phi(0)=\det\Id=1$. Consequently, $0<\phi(\lambda)$ for every $0\leq\lambda<1$, which, in turn, implies that $\det A=\phi(1)=\lim_{\lambda\to1^-}\phi(\lambda)\geq 0$. 
\end{proof}

We now focus our attention on $ML(n)$. To this end, we first recall the case where a mapping $A\in MLR(n)$ is, in fact, in $ML(n)$. Such a characterization is the combination of \cite[Fact 2.2]{bwy2012} and \cite[Fact 2.3]{bwy2012}:

\begin{fact}
For $A\in MLR(n)$ the following assertions are equivalent:
\begin{enumerate}
\item $A\in ML(n)$;
\item $A$ is at most single-valued;
\item $A(0)$ is a singleton;
\item $\dom A=\RR^n$.
\end{enumerate}
\end{fact}

In our current discussion it is crucial to distinguish between the notion of the inverse of a mapping in $MLR(n)$ (multivalued settings) and the classical notion of the inverse of a single-valued mapping in $MLI(n)$. In the case where $A\in MLI(n)$, we will abuse the notation and write $A^{-1}$ for its inverse, which can then be viewed as the same in both settings, the multivalued and the classical single-valued. In order for our results in the present paper to be relevant for studies within classical linear algebra settings in $ML(n)$, we now support our claim that:
\begin{quotation}
When taking the resolvent average of operators in ML(n), all of the inversion operations  involved in the averaging operation are classical inversions. Furthermore, for operators in MLI(n), all of the inverses in the formula $(\average)^{-1}=\averageinverse$ are classical inverses.
\end{quotation}
Indeed, the ``Furthermore" part follows from the invertibility of the $A_i$'s when we recall Theorem~\ref{classical linearity dominance within}. Now, if $0<\mu$ and $A_i\in ML(n)$, then $\mu A_i+\Id$ is invertible (since $0<\|x\|^2\leq\scal{x}{(\mu A_i+\Id)x}$ for every $0\neq x\in\RR^n$). Thus, since $A_i$ is monotone, then $J_{\mu A_i}$ is now seen to be invertible and firmly nonexpansive, that is, for any $x\in\RR^n,\ x\neq0$, we have $0<\|J_{\mu A_i}x\|^2\leq\scal{x}{J_{\mu A_i}x}$. As a consequence, 
$$
\scal{x}{J_{\mu\average}x}=\sum_{i\in I}\lambda_i\scal{x}{J_{\mu A_i}x}>0.
$$
Since this is true for every $x\in\RR^n,\ x\neq0$, we conclude that $J_{\mu\average}$ is, indeed, invertible, as asserted.

We now focus our attention further on $ML(n)\cap O(n,\RR)$. Here $O(n,\RR)$ is the group of orthogonal matrices in $\RR^{n\times n}$; in particular, $A\in O(n\RR)\Rightarrow\det A\pm 1$. Since, as we saw in Proposition~\ref{det}, the determinant is greater than zero in $ML(n)$, then, in fact, $ML(n)\cap O(n,\RR)=ML(n)\cap SO(n,\RR)$, where $SO(n)$ is the group of the special orthogonal matrices, that is, the matrices $A\in O(n)$ such that $\det A=1$. Thus, $ML(n)\cap O(n,\RR)=ML(n)\cap SO(n,\RR)$ can be viewed as the subset of rotations which consists of the rotations by an acute or right angles. The next result demonstrates that the resolvent average of such rotations is again such a rotation. This, of course, fails when taking the arithmetic average.

\begin{theorem}\label{rotations}
\textbf{\emph{(being a rotation by an acute or right angle is recessive)}}
Suppose that for each $i\in I$, $A_i\in ML(n)\cap O(n,\RR)=ML(n)\cap SO(n,\RR)$. Then $\averageonelambda \in ML(n)\cap O(n,\RR)=ML(n)\cap SO(n,\RR)$. 
\end{theorem}

\begin{proof}
By employing the inversion formula~\eqref{mainresult} (which, in this case, was seen to be a classical inversion), we obtain
$$
\big(\averageonelambda\big)^{-1} = \averageoneinverse
=\mathcal{R}({\bf A}^{\intercal},\lambda)
=\averageonelambda^{\intercal}.
$$
That is, $\averageonelambda$ is orthogonal. The fact that $\det\averageonelambda=1$ follows from the monotonicity of $\averageonelambda$ (see Proposition~\ref{det}).
\end{proof}

\begin{example}[generating Pythagorean triples]
When $n=2$, we consider the case where $A_1=\Id$, $A_2$ is the counter clockwise rotation by $\frac{\pi}{2}$, $0<\lambda<1$, $\lambda_1=\lambda$ and $\lambda_2=1-\lambda$. Then
$$
\averageonelambda = \frac{1}{\lambda^2-2\lambda+2}\begin{pmatrix}\lambda(2-\lambda)  & -2(1-\lambda)\\2(1-\lambda) & \lambda(2-\lambda)\end{pmatrix}
$$
is a counter-clockwise rotation matrix by an angle of a right triangle with sides 
$$
a(\lambda)=\lambda(2-\lambda),\ \ b(\lambda)=2(1-\lambda)\ \ \ \ \text{and}\ \ \ \ c(\lambda)=\sqrt{a^2+b^2}=\lambda^2-2\lambda+2.
$$
Now $\lambda\mapsto\averageonelambda$ is a smooth and one-to-one curve, in particular, $\lambda\mapsto a(\lambda)/c(\lambda)$ is a bijection of $[0,1]$ to itself. (In fact, for any two monotone mappings $A_1:\HH\rightrightarrows\HH$ and $A_2:\HH\rightrightarrows\HH$ such that $A_1\neq A_2$, $\lambda\mapsto\averageonelambda$ is one-to-one). Consequently, up to rescaling, all possible right triangles can be recovered in this manner. In particular, considering all of the possible rational values of $\lambda$, our procedure recovers all possible \emph{primitive Pythagorean triples} $(a,b,c)$ (other values of $\lambda$ should be considered as well in order to recover all possible Pythagorean triples). Indeed, letting $\lambda=\frac{p}{q}$ where $0<p<q$ are natural numbers, we obtain the matrix
$$
\averageonelambda=
\frac{1}{\lambda^2-2\lambda+2}\begin{pmatrix}\lambda(2-\lambda) & -2(1-\lambda)\\2(1-\lambda) & \lambda(2-\lambda)\end{pmatrix}
 =\frac{1}{p^2-2pq+2q^2}\begin{pmatrix}
p(2q-p) & -2q(q-p)\\2q(q-p)& p(2q-p)
\end{pmatrix}
$$  
which is a counter-clockwise rotation matrix by an angle of a right triangle with sides 
$$
a=p(2q-p),\ \ b=2q(q-p)\ \ \ \ \text{and}\ \ \ \ \ c=p^2-2pq+2q^2.
$$
Letting $k=q-p$ and $l=q$, we obtain the formula
\begin{equation}\label{Euclid's formula}
a=l^2-k^2,\ \ b=2kl\ \ \ \ \ \text{and}\ \ \ \ c=k^2+l^2,
\end{equation}
which is a well known formula for generating all of the primitive Pythagorean triples and is  attributed to Euclid. For further, more accurate, relations between formula~\eqref{Euclid's formula} and primitive Pythagorean triples as well as for an extensive historical overview  see~\cite[Section 4.2]{MathExp}.
\end{example}

\subsection{Nonexpansive monotone operators and Banach contractions}
We continue our discussion of nonexpansive monotone operators. Within this class of mappings, being a Banach contraction is a dominant property:
   
\begin{theorem}\label{Banach contraction dominance within}
\textbf{\emph{(within the class of nonexpansive mappings, being a Banach contraction is dominant)}}
Suppose that for each $i\in I$, $A_i:\HH\to\HH$ is nonexpansive and monotone. If there exists $i_0\in I$ such that $A_{i_0}$ is a Banach contraction, then $\averageonelambda$ is a Banach contraction. 
\end{theorem}

\begin{proof}
\cite[Corollary 4.7 ]{bmw12} asserts that given a maximally monotone operator $B:\HH\rightrightarrows\HH$, letting $A=2J_B-\Id$, then $A$ is a Banach contraction if and only if $B$ and $B^{-1}$ are strongly monotone. We now employ the settings and the outcome of Theorem~\ref{nonexpansive average}: for every $i\in I$, $A_i:\HH\to\HH$ is monotone and nonexpansive, $A_i=2J_{B_i}-\Id$ for a maximally monotone and nonexpansive mapping $B_i$ and $\averageonelambda=2J_B-\Id$ where $B:\HH\to\HH$ is the monotone and nonexpansive mapping $B=\sum_{i\in I}\lambda_iB_i$. We also see that $B_{i_0}$ is strongly monotone, say with $\epsilon_{i_0}$ being its constant of strong monotonicity. Now, let $x$ and $y$ be points in $\HH$. Then
\begin{align*}
\scal{Bx-By}{x-y}&=\sum_{i\in I}\lambda_i\scal{B_i x-B_i y}{x-y}\geq\lambda_{i_0}\scal{B_{i_0}x-B_{i_0}y}{x-y}\\
&\geq\lambda_{i_0}\epsilon_{i_0}\|x-y\|^2\geq\lambda_{i_0}\epsilon_{i_0}\|Bx-By\|^2.
\end{align*}
Thus, we see that both, $B$ and $B^{-1}$ are $\lambda_{i_0}\epsilon_{i_0}$-strongly monotone, which, in turn, implies that $\averageonelambda=2J_B-\Id$ is a Banach contraction. 
\end{proof}

\subsection{Projections and normal cones}

We will say that a property $(p)$ is \textbf{\emph{indeterminate}} (with respect to the resolvent average) if $(p)$ is neither dominant nor recessive.

\begin{example}\label{201406121}
\textbf{(being a projection is indeterminate)}
Let $A_1$ and $A_2$ be the projections in $\RR^2$ onto $\RR \times \{0\}$ and $\{0\} \times \RR$, respectively. That is,
$$A_1=\begin{pmatrix}
1 & 0\\
0 & 0
\end{pmatrix}\text{ and }A_2=\begin{pmatrix}
0 & 0\\
0 & 1
\end{pmatrix}.
$$
Let $0<\lambda<1$, $\lambda_1=\lambda$ and $\lambda_2=1-\lambda$. Then 
$$
\averageonelambda = \begin{pmatrix}
\frac{\lambda}{2-\lambda} & 0\\[6 pt]
0 & \frac{\lambda-1}{\lambda+1}
\end{pmatrix},
$$
which is not a projection since $\averageonelambda^2\neq\averageonelambda$.
\end{example}

\begin{example}\label{201406122}
\textbf{(being a normal cone operator is indeterminate)}
Let $f_1:\RR^2\to\RX$ be the function $f_1=\iota_{\RR\times\{0\}}$ and let $f_2:\RR^2\to\RX$ be the function $f_1=\iota_{\{0\}\times\RR}$. Let $A_1:\RR^2\rightrightarrows\RR^2$ be the normal cone operator $A_1=N_{\RR\times\{0\}}=\partial f_1$ and let $A_2:\RR^2\rightrightarrows\RR^2$ be the normal cone operator $A_2=N_{\{0\}\times\RR}=\partial f_2$. Let $0<\lambda<1$, $\lambda_1=\lambda$ and $\lambda_2=1-\lambda$, then
$$J_{A_1}=\begin{pmatrix}
1 & 0\\
0 & 0
\end{pmatrix}\text{, }J_{A_2}=\begin{pmatrix}
0 & 0\\
0 & 1
\end{pmatrix}
\ \text{\ and\  }\ 
\averageonelambda = \begin{pmatrix}\frac{1-\lambda}{\lambda} & 0\\ 0 & \frac{\lambda}{1-\lambda}\end{pmatrix}.
$$
Thus, we see that $\averageonelambda=\partial p({\bf f},{\boldsymbol \lambda}) =\nabla p({\bf f},{\boldsymbol \lambda})$ where $p({\bf f},{\boldsymbol \lambda}):\RR^2\to\RR$ is the parabola $p({\bf f},{\boldsymbol \lambda})(x,y)=\frac{1-\lambda}{2\lambda}x^2+\frac{\lambda}{2(1-\lambda)}y^2$. Since the antiderivative of $\averageonelambda$ is unique up to an additive constant and since $p({\bf f},{\boldsymbol \lambda})\neq\iota_C$ for any subset $C$ of $\RR^2$, we see that $\averageonelambda$ is not of the form $\partial\iota_C$, that is, $\averageonelambda$ is not a normal cone operator.
\end{example}

\subsection{Lipschitz monotone operators}\label{Lipschitz}

The purpose of this last subsection is to point out that further study is still required in order to determine if Lipschitzness is a property which is dominant, recessive or indeterminate. Within certain classes of monotone operators we do, however, have several conclusions: (i) Within the class of monotone linear relations, Theorem~\ref{classical linearity dominance within}(i) guarantees that Lipschitzness is dominant w.r.t. the resolvent average. However, no quantitative result is currently available with regard to the Lipschitz constant. (ii) Theorem~\ref{nonexpansive average} guarantees that 1-Lipschitzness is recessive in the case where $\mu=1$. However, no such result is available for other Lipschitz constants. Furthermore, within the class of nonexpansive monotone mappings, we saw that being a Banach contraction is dominant, as asserted by Theorem~\ref{Banach contraction dominance within}, still, with no quantitative information regarding the Lipschitz constant. (iii) Within the class of subdifferential operators we do have dominance of Lipschitzness w.r.t. the resolvent average with an explicit Lipschitz constant, namely, Theorem~\ref{Lip grad}. However, as implied by Example~\ref{strong sharpness counter for subdifferentials}, our explicit Lipschitz constant is not.

Let $\alpha>0$. In the following example, within the class of monotone linear operators and outside the class of subdifferential operators, we take the resolvent average of two mappings with a common sharp Lipschitz constant $\alpha$ such that their resolvent average has a sharp Lipschitz constant $\alpha^2$. This is in stark contrast to the constant in Theorem~\ref{Lip grad} for subdifferential operators.   
\begin{example}\label{scaled rotations}
Let 
$A_1=\begin{pmatrix}
0 & -\alpha\\
\alpha & 0
\end{pmatrix}=A_2^{\intercal}
$.
Then
$$ 
J_{A_1}=\frac{1}{\alpha^2+1}\begin{pmatrix}
1 & \alpha\\-\alpha & 1
\end{pmatrix}=J_{A_2}^{\intercal}\ \ \text{ and }\ \ \ 
\mathcal{R}(\bf{A})=\begin{pmatrix}
\alpha^2 & 0\\
0 & \alpha^2
\end{pmatrix}.
$$
We see that, indeed, the sharp Lipschitz constant of $A_1$ and $A_2$ is $\alpha$ and the sharp Lipschitz constant of $\mathcal{R}(\bf{A})$ is $\alpha^2$.
\end{example}

\section*{Acknowledgments}
Sedi Bartz was supported
by a postdoctoral fellowship of the Pacific Institute for the Mathematical Sciences and by NSERC
grants of Heinz Bauschke and Xianfu Wang.
Heinz Bauschke was partially supported by the Canada Research
Chair program and by the Natural Sciences and Engineering Research
Council of Canada.
Sarah Moffat was partially supported by the Natural Sciences and
Engineering Research Council of Canada.
Xianfu Wang was partially supported by the Natural Sciences and
Engineering Research Council of Canada.


\begin{thebibliography}{}

\bibitem{AusTeb}
A.\ Auslender and M.\ Teboulle, \emph{Asymptotic Cones and
Functions in Optimization and Variational Inequalities},
Springer-Verlag, 2003. 

\bibitem{BBBRW}
S.\ Bartz, H.H.\ Bauschke, J.M.\ Borwein, S.\ Reich and X.\
Wang, ``Fitzpatrick functions, cyclic monotonicity and Rockafellar's antiderivative",  \emph{Nonlinear Analysis: Theory, Methods and Applications}~66,
pp.~1198--1223, 2007.

\bibitem{BB1997} H.H. Bauschke and J.M. Borwein, ``Legendre functions and the method of random Bregman projections", \emph{Journal of Convex Analysis}~4(1), pp.~27--67, 1997.

\bibitem{bbw2007}
H.H. Bauschke, J.M. Borwein and X. Wang,
``Fitzpatrick functions and continuous linear monotone operators", \emph{SIAM Journal on Optimization}~18,
pp.~789--809, 2007.

\bibitem{BC2011}
H.H.\ Bauschke and P.L.\ Combettes,
\emph{Convex Analysis and Monotone Operator Theory in Hilbert Spaces},
Springer, 2011.

\bibitem{BGLW}
H.H.\ Bauschke, R.\ Goebel, Y.\ Lucet and X.\ Wang, ``The proximal average: basic theory", \emph{SIAM Journal on Optimization}~19, pp.~766--785, 2008.

\bibitem{BM}
H.H\ Bauschke, W.L. Hare and W.M. Moursi, ``On the range of the Douglas-Rachford operator'', \emph{preprint}, 2014.\ \ \ \texttt{http://arxiv.org/abs/1405.4006} 

\bibitem{BLT}
H.H.\ Bauschke, Y.\ Lucet and M.\ Trienis, ``How to transform one convex function continuously into another", \emph{SIAM Review} 50, pp. 115--132, 2008. 

\bibitem{BMR}
H.H.\ Bauschke, E.\ Matou\v{s}kov\'{a} and S.\ Reich, ``Projection and proximal point methods: convergence results and counterexamples", \emph{Nonlinear Analysis: Theory, Methods and Applications}~56(5), pp. 715--738, 2004.


\bibitem{bmow2013} H.H.\ Bauschke, S.M.\ Moffat and X.\ Wang, ``Near equality, near convexity, sums of maximally monotone operators, and averages of firmly nonexpansive mappings", \emph{Mathematical Programming}~139, pp.~55--70, 2013.

\bibitem{bmw12} H.H.\ Bauschke, S.M.\ Moffat and X.\ Wang, ``Firmly nonexpansive mappings and maximally monotone operators: correspondence and duality", \emph{Set-valued and Variational Analysis}~20,  pp.~131--153, 2012.

\bibitem{bmw-res}
H.H.\ Bauschke, S.M.\ Moffat and X.\ Wang, ``The resolvent average for positive semidefinite matrices", \emph{Linear Algebra and Its Applications}~432, pp.~1757--1771, 2010.

\bibitem{BW}
H.H.\ Bauschke and X.\ Wang, ``The kernel average for two convex functions and its applications to the extension and representation of monotone operators", \emph{Transactions of the American Mathematical Society}~361, pp.~5947--5965, 2009.

\bibitem{bwy2012}
H.H.\ Bauschke, X.\ Wang and L.\ Yao, ``Rectangularity and
paramonotonicity of maximally monotone operators",
\emph{Optimization}~63, pp.~487--504, 2014.

\bibitem{Borwein}
J.M.\ Borwein,
``Fifty years of maximal monotonicity'',
\emph{Optimization Letters}~4, pp.~473--490, 2010. 

\bibitem{BV}
J.M.\ Borwein and J.D.\ Vanderwerff, \emph{Convex Functions:
Constructions, Characterizations and Counterexamples}, Cambridge University Press, 2010.

\bibitem{BH}
H.\ Br\'{e}zis and A.\ Haraux, ``Image d'une somme d'op\'erateurs monotones
et applications", \emph{Israel Journal of Mathematics}~23, pp.~165--186, 1976.

\bibitem{BurIus}
R.S.\ Burachik and A.N.\ Iusem,
\emph{Set-Valued Mappings and Enlargements of Monotone Operators},
Springer-Verlag, 2008.


\bibitem{Cross}
R. Cross, \emph{Multivalued Linear Operators}, Marcel Dekker, 1998.

\bibitem{EckBer}
J.\ Eckstein and D.P.\ Bertsekas,
``On the Douglas-Rachford splitting method
and the proximal point algorithm for maximal monotone
operators",
\emph{Mathematical Programming Series A} 55, pp.~293--318, 1992.

\bibitem{Fitz}
S.\ Fitzpatrick, ``Representing monotone operators by convex functions", \emph{Workshop/Miniconference on Functional Analysis and Optimization (Canberra 1988)}, Proceedings of the Centre for Mathematical Analysis, Australian National University vol. 20, Canberra, Australia, pp.~59--65, 1988.

\bibitem{GHW}
R.\ Goebel, W.\ Hare and X.\ Wang,
``The optimal value and optimal solutions of the proximal average
of convex functions'',
\emph{Nonlinear Analysis}~75, pp.~1290--1304, 2012. 

\bibitem{Gho}
N.\ Ghoussoub, \emph{Self-dual Partial Differential Systems and Their Variational Principles}, Springer, 2009.

\bibitem{MathExp}
R. Laubenbacher and  D. Pengelley, \emph{Mathematical Expeditions: Chronicles by the Explorers}, Springer--Verlag, New York, 1999.

\bibitem{Minty}
G.J.\ Minty,
``Monotone (nonlinear) operators in Hilbert spaces",
\emph{Duke Mathematical Journal}~29, pp.~341--346, 1962.

\bibitem{Moffat}
S.M. Moffat, \emph{The Resolvent Average : An Expansive Analysis of Firmly Nonexpansive Mappings and Maximally Monotone Operators}, PhD Dissertation, The University of British Columbia, Okanagan, 2015.\ \ \ \texttt{http://circle.ubc.ca/handle/2429/51593}

\bibitem{Mor}
J.-J. Moreau, ``Proximit\'{e} et dualit\'{e} dans un espace hilbertien", \emph{Bulletin de la Soci\'{e}t\'{e} Math\'{e}matique de France}, 93, pp. 273--299, 1965.

\bibitem{Pen}
T.\ Pennanen, ``On the range of monotone composite mappings", \emph{Journal of Nonlinear and Convex Analysis}~2(2), pp.~193--202, 2001.

\bibitem{Reich 1983}
S.\ Reich, ``A limit theorem for projections", \emph{Linear and Multilinear Algebra}, 13, pp.~281--290, 1983.

\bibitem{Rock cyc}
R.T.\ Rockafellar, ``Characterization of the subdifferentials of convex functions", \emph{Pacific Journal of Mathematics} 17, pp.~497--510, 1966.

\bibitem{Rock bndd}
R.T.\ Rockafellar, ``Local boundedness of nonlinear, monotone operators", \emph{Michigan Mathematical Journal} 16, pp.~397--407, 1969. 

\bibitem{Rock} R.T.\ Rockafellar, \emph{Convex Analysis}, Princeton University Press, Princeton, NJ, 1970.

\bibitem{RockWets}
R.T.\ Rockafellar and R.J-B\ Wets,
\emph{Variational Analysis},
Springer-Verlag, 1998.

\bibitem{Simons2}
S.\ Simons, \emph{From Hahn-Banach to Monotonicity},
 Lecture Notes in Mathematics, Vol. 1693,
Springer-Verlag, 2008.

\bibitem{Yu}
Y.\ Yu, ``Better approximation and faster algorithm using the proximal average",
\emph{Advances in Neural Information Processing Systems 26 , (NIPS 2013)}, C.J.C.\ Burges, L.\ Bottou, M.\ Welling, Z.\ Ghahramani and K.Q. Weinberger (editors), Curran Associates, Inc. pp.~458--466, 2013.

\bibitem{Wang}
X.\ Wang, ``Self-dual regularization of monotone operators via the resolvent average", \emph{SIAM Journal on Optimization}~21, pp.~438--462, 2011.

\bibitem{Zal} C. Z\u{a}linescu, \emph{Convex Analysis in General Vector spaces}, World Scientific Publishing Co, 2002.

\end{thebibliography}
\end{document}